\documentclass[final,12pt]{colt2025} 

\usepackage{notations}

\usepackage{bbm}

\usepackage{xcolor}
\usepackage{caption}
\usepackage{tikz}
\usetikzlibrary{arrows,decorations.pathmorphing,backgrounds,positioning,fit,petri}
\usepackage{bbm}
\usepackage{csquotes}
\usepackage{comment}
\hypersetup{breaklinks=true}

\definecolor{RossoE}{RGB}{198, 0, 0}%

\title[Reduction of deep neural networks in the overparametrized regime]{Information-theoretic reduction of deep neural networks \\to linear models in the overparametrized proportional regime}
\usepackage{times}

\coltauthor{%
\Name{Francesco Camilli} \Email{fcamilli@ictp.it} \\
       \addr The Abdus Salam International Centre for Theoretical Physics\\
       Trieste 34151, Italy
       \AND
       \Name{Daria Tieplova} \Email{dtieplov@ictp.it} \\
       \addr The Abdus Salam International Centre for Theoretical Physics\\
       Trieste 34151, Italy
       \AND
       \Name{Eleonora Bergamin} \Email{ebergami@sissa.it} \\
       \addr International School for Advanced Studies\\
       Trieste 34136, Italy
       \AND
       \Name{Jean Barbier} \Email{jbarbier@ictp.it} \\
       \addr The Abdus Salam International Centre for Theoretical Physics\\
       Trieste 34151, Italy.
}

\begin{document}

\maketitle

\begin{abstract}%
We rigorously analyse fully-trained neural networks of arbitrary depth in the Bayesian optimal setting in the so-called \emph{proportional scaling regime} where the number of training samples and width of the input and all inner layers diverge proportionally. We prove an information-theoretic equivalence between the Bayesian deep neural network model trained from data generated by a teacher with matching architecture, and a simpler model of optimal inference in a generalized linear model. This equivalence enables us to compute the optimal generalization error for deep neural networks in this regime. We thus prove the ``deep Gaussian equivalence principle'' conjectured in \cite{cui2023optimal}. Our result highlights that in order to escape this ``trivialisation'' of deep neural networks (in the sense of reduction to a linear model) happening in the strongly overparametrized proportional regime, models trained from much more data have to be considered.
\end{abstract}


\section{Introduction}

Neural networks are exceptional tools in many applications such as image classification and speech processing to name a few \citep{Alzubaidi2021ReviewOD}. They have therefore stimulated an unprecedented interest in their behavior, which is still understood mostly at the empirical level. The reason is that developing a quantitative theory for their performance is a challenging mathematical task. The challenge resides in the complex interaction of at least three aspects: \emph{i)} their complex architecture; \emph{ii)} the impact of the data structure; \emph{iii)} the complex dynamics of the optimisation algorithms used for training. In the present paper, we shall focus on information-theoretical limits, which are algorithm-independent, and structureless input data, in order to highlight the phenomenology induced by the architecture alone, point \emph{i)}.
In order to study the fundamental limits of learning  a certain target function, the correct framework to consider is the \emph{teacher-student scenario}, where said target function plays the role of the ``teacher'' that generates the responses associated with the inputs. Then, a \enquote{student} model belonging to the same hypothesis class must learn the target based on these examples. In the present setting this entails that the teacher and student neural networks will have the very same architecture and hyper-parameters, and only the realization of the teacher's weights is to be learnt.
 The importance of the teacher-student scenario was pinpointed in the seminal work of \cite{Gardner-derrida-perceptron} on the perceptron. 

The Bayes-optimal generalization error for the perceptron and generalized linear model was rigorously computed in \cite{JeanGLM_PNAS}. This model and the committee machine have been intensively studied since the nineties. We refer the reader to the (non-exhaustive) list of contributions: \cite{barkai1992broken,Zippelius_committee_92,schwarze1993learning,schwarze1992generalization,schwarze1993generalization,monasson1995weight,Saad_Solla_committee_95,mato1992generalization,engel2001statistical,aubin2018committee,baldassi2019properties,goldt2020dynamics,goldt2020hiddenmanifold,reeves_MarcGET}. For these models where the number of hidden units is dimension-independent, the relevant interpolation regime corresponds to a number of samples diverging proportionally to the input dimension.

Concerning settings with a large number of hidden units we have the series of works \cite{nitanda2017stochastic,sirignano2020mean,araujo2019mean,nguyen2019mean,nguyen2020rigorous,mean-field-landscape-2-layer,mean-field-2-layer-kernal-limit,Chizat-Bach-2018,Rostkoff-Eijnded-23,Marco-Diyuan22,Marco-Montanari-AOS-MF,Marco-Alex-MF-piecewise,pmlr-shevchenko20a,mean-field-2-layer-kernal-limit}, investigating online learning with SGD. Mean-field analyses of SGD differ from the information-theoretical one. SGD produces a ``one-shot estimator'' for labels, which is in general sub-optimal. Furthermore, in online learning, the student network sees one sample at a time, whereas in Bayes-optimal learning the whole dataset is jointly exploited. 

As evidenced by \cite{mean-field-2-layer-kernal-limit}, when few samples are provided to the neural network during training, the weights are virtually fixed to their initialization, which is the setting captured by the Neural Tangent Kernel of \cite{jacot-gabriel-hongler-2021} and similarly studied in random feature models and lazy training regimes \citep{RF-original,NIPS2017_a96b65a7,RF-rudi-rosasco,Bach-2017,li2018learning,allen2019convergence,du2018gradient,du2019gradient,lee2020wide,arora2019exact,huang2020dynamics,MontanariMei-RF-regression,ghorbani2019linearized, montanari2020interpolation,gerace2020GET_ICML,d2020double,geiger2020perspective,dhifallah2020precise,bordelon2020spectrum,Yue2022GEP,Dmitriev_deep_RF,bosch2023precise,bruno2023det_equivalent}.

In \cite{SompolinskyPRX2022} the authors instead tackled the learning problem of all the weights for deep linear networks whose width is proportional to the input dimension, see also \cite{ZavatoneVeth2022ContrastingRA}. Their linear nature allows to solve the statistical mechanics of these networks by a subsequent integration of the degrees of freedom associated with the network weights. A whole subsequent line of work generalised their approach to tackle non-linear networks in the so-called \emph{proportional regime}, in which the number of training examples $n$ and the sizes of all layers $(d_\ell)$ diverge simultaneously and in a proportional way: \cite{pacelli2023,naveh2021,seroussi2023,aiudi2023,bassetti2024,rubin2025kernelsfeaturesmultiscaleadaptive}. In this strongly overparametrized proportional regime, these works show that \emph{some} feature learning happens, in the sense that the kernel learnt by the network adapts to the data w.r.t. to the fixed kernel that emerges in the infinite width limit described by the neural network Gaussian processes~\citep{neal1996,williams1996,lee2018gaussian,matthews2018gaussian,hanin2023infinite}. Still in the proportional regime but closer to our setting, \cite{cui2023optimal} conjecture a formula for the Bayes-optimal generalization error of a deep neural network using the replica method \citep{mezard1987spin}. One key ingredient of their analysis is a ``deep Gaussian equivalence principle'' conjecture, a version of Gaussian Equivalence Principle (GEP). Classical results support the validity of GEPs in high-dimensional inference  \citep{sudakov1978proj,diaconis1984asymptotics,reeves2017conditionalCLT,meckes2012approximation}, such as in the description of relevant observables for shallow neural networks \citep{goldt2020hiddenmanifold,reeves_MarcGET}. 

In the present paper, we prove this deep Gaussian equivalence principle for the proportional scaling regime (and slightly beyond) and, as a consequence, the associated replica symmetric formula for the free entropy and generalisation error conjectured in \cite{cui2023optimal}. Doing so, we substantially improve over \cite{camilli-tiepova-barbier2023fundamental} which was limited to a shallow neural network and did not reach the same convergence rates. The proof follows an induction strategy: in the considered scaling regime, we are able to show that a network with $L$ layers reaches the same Bayes-optimal generalization error as a suitably tuned neural network with $L-1$ layers. In order to achieve this, we are required to prove that the non-linearity can be replaced by an \enquote{equivalent} linearity, which amounts to prove a Gaussian Equivalence Principle.

Let us finally mention the analyses of the multi-layer generalised linear model of \cite{three-layer-GLM_Jean_Marylou,chen2023free}. Although in this model the weights are fixed, and thus not learnable as in our setting, it shares some common challenges due to the multi-layer structure, especially related to the concentration properties w.r.t.\ the quenched randomness of the model.

\paragraph{Notations.} Bold notations are reserved for vectors and matrices. A vector $\bx$ is a column vector; $\bx^\intercal$ is then a row vector. The $L_2$ norm is $\|\bx\|^2=\bx^\intercal \bx$; $\bx\bx^\intercal$ is a rank-one projector. $\EE_A$ is an expectation with respect to the r.v.\ $A$; $\EE$ is an expectation with respect to all ensuing r.v.\ 's. $\mathbb{V}(A)$ is the variance of $A$, and $\mathbb{V}_A(\,\cdot\,)$ the variance w.r.t. $A$ only conditional on the rest. For a function $F$ of one argument we denote $F'$ its derivative. $[N]\equiv \{1,2,\dots,N\}$. We write $\EE(\cdots)^2=\EE[(\cdots)^2]\ge (\EE(\cdots))^2=\EE^2(\cdots)$. We shall abbreviate $\EE[f(\cdots)]\equiv\EE f(\cdots)$. For any functions $f$ and $g$, $f=O(g)$ means there exists a constant $K$ such that $|f|\leq K|g|$. Hence, given a r.v.\ $Z$, $f=O(Z)$ does \emph{not} imply $\E f =0$, but rather $|\E f|\leq \E|f|\leq K\E|Z|$. $\mathbbm{1}_k$ is the identity matrix in $k$ dimensions. Starred variables, e.g. $\ba^*$, are the teacher's weights. Non-starred variables, e.g. $\ba$, will denote the student's weights. $X\sim P$ means that the r.v. is drawn from, or has law, $P$.

\section{Definitions and main results}
We consider a teacher-student setting, where the training set, namely the set of $n$ input-response couples $\mathcal{D}_L=\{(\bX_\mu^{(0)},Y_\mu)\}_{\mu=1}^n$, is generated by a teacher neural network with $L$ layers fed with structureless inputs $\bX_\mu^{(0)}\iid\mathcal{N}(0,\mathbbm{1}_{d_0})$. Responses $Y_\mu$ are assumed to be real numbers drawn from an output channel accounting for possible randomness in the output:
\begin{align}\label{eq:output_channel_labels}
    Y_\mu\sim P_{  out}\Big(\cdot\mid \rho\frac{\mathbf{a}^{*\intercal}\bX_\mu^{(L)}}{\sqrt{d_L}}+\sqrt{\epsilon}\xi_\mu^*\Big)\,,\quad \ba^*\in\mathbb{R}^{d_L}\,,\,\xi_\mu^*\iid\mathcal{N}(0,1)
\end{align}
where we have introduced the post-activations
\begin{align}
    \bX_\mu^{(\ell)}=\varphi\Big(\frac{\bW^{*(\ell)}\bX^{(\ell-1)}_\mu}{\sqrt{d_{\ell-1}}}\Big)\in\mathbb{R}^{d_\ell}\,,\quad  \,\bW^{*(\ell)} \in\mathbb{R}^{d_\ell\times d_{\ell-1}}\,,\,\ell\in[L]\,,
\end{align}with the non-linearity $\varphi$ applied component-wise. $\rho> 0$, and $\epsilon\ge 0$ can be taken to be zero and is introduced here for later convenience. The positive integers $(d_\ell)_{\ell=0}^L$ are thus the dimensions of the layers. We group the teacher's weights in $\btheta^{*(L)}:=(\ba^*,(\bW^{*(\ell)})_{\ell=1}^L)$, and we assume all of them are i.i.d.\ drawn from standard Gaussians:
\begin{align}
    a^*_i\,,\,W^*_{jl}\iid P_\theta\equiv\mathcal{N}(0,1)\,,\quad
    dP(\btheta^{*(L)}):=\prod_{i=1}^{d_L}dP_\theta(a^*_i)\prod_{\ell=1}^L\prod_{j=1}^{d_\ell}\prod_{k=1}^{d_{\ell-1}}dP_\theta(W^{*(\ell)}_{jk})\,.
\end{align}
A convenient representation for \eqref{eq:output_channel_labels} is the following:
\begin{align}\label{eq:output_Gaussian_channel}
    Y_\mu=f\Big(\rho\frac{\mathbf{a}^{*\intercal}\bX_\mu^{(L)}}{\sqrt{d_L}}+\sqrt{\epsilon}\xi_\mu^*;\bA_\mu\Big)+\sqrt{\Delta}Z_\mu\,,\quad Z_\mu\iid\mathcal{N}(0,1)\,.
\end{align}Here $\bA_\mu\in\mathbb{R}^p$, with $p$ a given integer, is some possible stochasticity in the output and $\Delta$ is the strength of the label noise. $f:\mathbb{R}\mapsto\mathbb{R}$ is the readout function. The mapping between \eqref{eq:output_channel_labels} and \eqref{eq:output_Gaussian_channel} is realized by the identification
\begin{align}\label{eq:output_kernel_vs_Gauss_channel}
    P_{  out}(y\mid x)=\int_{\mathbb{R}^p}\frac{dP_A(\bA)}{\sqrt{2\pi\Delta}}\exp\Big[-\frac{1}{2\Delta}(y-f(x;\bA))^2\Big]\,.
\end{align}Therefore, suitable regularity hypothesis on $f$ will also reflect in the output kernel $P_{  out}$.

The Bayes-optimal student network learns from $\mathcal{D}_L$ and estimates the response associated with a new input $\bX_{  new}\in\mathbb{R}^{d_0}$ as
\begin{align}\label{eq:BO_label_estimate}
    \hat{Y}_{  new}=\int y\,\EE\Big[P_{  out}\Big(y\mid \frac{\mathbf{a}^{\intercal}\bx_{  new}^{(L)}}{\sqrt{d_L}}+\sqrt{\epsilon}\xi_{new}\Big)\mid\mathcal{D}_L,\bX_{  new}\Big]\, dy
\end{align}where $\xi_{new}\sim\mathcal{N}(0,1)$, and we have introduced the student's post activations
\begin{align}
    \bx_{  new}^{(\ell)}=\varphi\Big(\frac{\bW^{(\ell)}\bx^{(\ell-1)}_{  new}}{\sqrt{d_{\ell-1}}}\Big)\in\mathbb{R}^{d_\ell}\,,\quad  \,\bW^{(\ell)} \in\mathbb{R}^{d_\ell\times d_{\ell-1}}\,,\,\ell\in[L]\,,
\end{align}with $\bx_{  new}^{(0)}\equiv\bX_{new}$. Analogous relations hold for the student's post activations $\bx_\mu^{(\ell)}$ associated with inputs in the training set. The measure $\mathbb{E}[\,\cdot\mid\mathcal{D}_L]$ that enters \eqref{eq:BO_label_estimate} is the Bayes' posterior measure, and it shall be denoted also by $\langle\,\cdot\,\rangle^{(L)}$. The estimator \eqref{eq:BO_label_estimate} provably yields (see \cite{JeanGLM_PNAS,camilli-tiepova-barbier2023fundamental}) the minimum (average) generalization error
\begin{align}\label{eq:gen_error_definition}
    \mathcal{E}^{(L)}:=\EE(Y_{  new}-\hat Y_{  new})^2
\end{align}where the expectation is w.r.t.\ to the law of the training set $\mathcal{D}_L$ and the test data $(\bX_{  new},Y_{new})$. 

We bypass the computation of the Bayes-optimal generalization error \eqref{eq:gen_error_definition} by computing the limiting Mutual Information (MI) between the training set $\mathcal{D}_L$, and the teacher's weights $\btheta^{*(L)}:=(\ba^*,(\bW^{*(\ell)})_{\ell=1}^L)$ and the variables $\bxi^*=(\xi^*_\mu)_{\mu=1}^n$. The MI can be written in a convenient way, in terms of the normalization of the Bayes posterior
\begin{align}\label{eq:Posterior}
    dP(\btheta^{(L)},\bxi\mid\mathcal{D}_L)= \frac{1}{\cZ_L(\mathcal{D}_L)} dP(\btheta^{(L)})\prod_{\mu=1}^n dP_\theta(\xi_\mu) P_{out}\Big(Y_\mu\mid\rho\frac{\ba^{\intercal}\bx_\mu^{(L)}}{\sqrt{d_L}}+\sqrt{\epsilon}\xi_\mu\Big)
\end{align}that is
\begin{align} \label{eq:partition_function}
    \cZ_L(\mathcal{D}_L)= \int\prod_{\mu=1}^n \EE_{\xi} P_{  out} \Big(Y_\mu\mid\rho\frac{\ba^{\intercal}\bx_\mu^{(L)}}{\sqrt{d_L}}+\sqrt{\epsilon}\xi_\mu\Big)dP(\btheta^{(L)})\,.
\end{align}With these notations the MI per data point rewrites as
\begin{align}\label{eq:MI_definition}
    \frac{1}{n}I_n^{(L)}(\btheta^{*(L)},\bxi^*;\mathcal{D}_L)=\EE\log  P_{  out}\Big(Y_1\mid \rho\frac{\ba^{*\intercal}\bX^{(L)}_1}{\sqrt{d_L}}+\sqrt{\epsilon}\xi^*\Big)-\frac{1}{n}\EE\log\mathcal{Z}_L(\mathcal{D}_L)\,.
\end{align}
The last term is called \emph{free entropy} and shall be denoted $\bar f^{(L)}_n:=\frac{1}{n}\EE\log\mathcal{Z}_L(\mathcal{D}_L)$. In addition, the MI with $L=0$, which means in the absence of hidden layers, has to be interpreted as the one of a Generalized Linear Model (GLM), see \cite{JeanGLM_PNAS}. We can now state our main result.
\begin{theorem}[Layer reduction] \label{thm:reduction}
    Let the activation and readout functions $\varphi,f\in C^2(\mathbb{R})$ be odd and with bounded first and second derivatives. Let $Z$ be a standard Gaussian, assume $L\geq 1$ and define the coefficients
    \begin{align}
        \sigma_0=1\,,\, \sigma_\ell=\EE\varphi^2(Z\sqrt{\sigma_{\ell-1}})\,,\quad \rho_\ell=\EE\varphi'(Z\sqrt{
        \sigma_{\ell-1}}),\quad \epsilon_\ell=\sigma_\ell-\sigma_{\ell-1}\rho_\ell^2\,,
    \end{align}for all $\ell\in[L]$.
    Denote by $d_m=\min\{(d_\ell)_{\ell=0}^L\}$. Then
    \begin{align}\label{eq:diff_MIs}
        \Big|\frac{1}{n}I_n^{(L)}(\btheta^{*(L)},\bxi^*;\mathcal{D}_L)-
        \frac{1}{n}{I}_n^{(L - 1)}(\btheta^{*(L-1)},\bxi^*;\mathcal{D}_{L-1})\Big|=O\Big(\Big(1+\sqrt{\frac{n}{d_m}}+\frac{n}{d_m}\Big)\frac{1}{\sqrt{d_m}}\Big)\,,
    \end{align}where ${I}_n^{(L - 1)}(\btheta^{*(L-1)},\bxi^*;\mathcal{D}_{L-1})$ is the mutual information between the hidden teacher weights, $\bxi^*$ and the data set $\mathcal{D}^{L-1}=\{(\bX_\mu^{(0)},Y_\mu^{(L-1)})\}_{\mu=1}^n$ with responses drawn as
    \begin{align}
        Y_\mu^{(L-1)}\sim P_{  out}\Big(\cdot\mid\rho\rho_L\frac{\ba^{*\intercal}\bX^{(L-1)}_\mu}{\sqrt{d_{L-1}}}+\sqrt{\rho^2\epsilon_L+\epsilon}\xi^{*}_\mu\Big)\,.\label{eq:reduced_channel}
    \end{align}
\end{theorem}
This Theorem shows that in the linear regime, at the information-theoretic level, layers can be eliminated one by one paying the price of extra Gaussian noise. This is the manifestation of the Gaussian Equivalence Principle in our setting. In fact, what the result asserts is that the non-linearity for the last layer can be replaced by a linear equivalent. The values of the coefficients $\epsilon_L,\rho_L$ that re-weight the two new terms appearing in the equivalent model with one less layer are the same found in \cite{MontanariMei-RF-regression,goldt2020hiddenmanifold,Yue2022GEP,camilli-tiepova-barbier2023fundamental}.

One key ingredient needed to prove this type of result is the concentration of the free entropy w.r.t. the teacher and data instance.  Even though this is a natural conjecture, it becomes particularly hard to prove in the multi-layer model due to the number of random parameters and their intertwined dependences. We thus state here this stand alone result which may be of independent interest.
\begin{theorem}[Free entropy concentration]\label{th:Z_k_concentr}
Consider the activation and readout function $\varphi,\,f$ as in Theorem~\ref{thm:reduction}. Then there exists a non-negative constant $C(f,\varphi)$ such that
\begin{align}
  \mathbb{V}\Big(\dfrac{1}{n}\log \cZ _L(\mathcal{D}_L)\Big)
    \leq C(f,\varphi)\Big(\frac{1}{n}+\frac{1}{d_m}\Big)\,.
\end{align}
\end{theorem}

The layer reduction Theorem can be iterated until exhaustion of the deep Bayesian network, yielding a Generalized Linear Model. During the reduction process one needs to keep track of how the re-weighting coefficients accumulate. After $L-k$ reduction steps, which means $k$ layers are still left, the network model will then be described by equation \eqref{eq:output_channel_labels}, where $\rho$ and $\epsilon$ are replaced by $\eta_k$ and $\gamma_k$ respectively, defined as follows:
\begin{align}
\begin{split}
&\eta_{k} =  \rho \prod_{i = k + 1}^{L}\rho_{i}\,, \qquad \gamma_{k} =\epsilon+ \rho^2 \sum_{j =k+1}^{L}\epsilon_{j} \prod_{i = j + 1}^{L}\rho^{2}_i  \, , \label{eq:coeffs}
\end{split}
\end{align}where by convention $\prod_{i=L+1}^L(\,\cdot\,)=1$, $\eta_L=\rho$ and $\gamma_L=\epsilon$. The coefficients in \eqref{eq:coeffs} match the ones in \cite{cui2023optimal}. The target model is the GLM where $k=0$, namely when all the $L$ hidden layers have been eliminated and only the readout layer remains. From our previous results, we have the following Corollary, proved in Appendix~\ref{app:complete_reduction}:

\begin{corollary}
[Mutual information equivalence]
    \label{thm:MIAllequivalence}
    Under the same assumptions as Theorem \ref{thm:reduction}, the following holds true:
    \begin{align}\label{boundMutualInfoAll}
    \begin{split}
        &\Big|\frac{1}{n}I_n^{(L)}(\btheta^{* (L)},\bxi^*;\mathcal{D}_L)-\frac{1}{n}I_n^{(0)}(\btheta^{* (0)},\bxi^*;\mathcal{D}_0)\Big| = O\Big(\Big(1+\sqrt{\frac{n}{d_m}}+\frac{n}{d_m}\Big)\frac{1}{\sqrt{d_m}}\Big) \, ,
    \end{split}
    \end{align}
    where ${I}_n^{(0)}(\btheta^{*(0)},\bxi^*;\mathcal{D}_{0})$ is the mutual information associated with the data set $\mathcal{D}$ with responses
    \begin{align*}
        Y_\mu^{(0)}\sim P_{  out}\Big(\cdot\mid\eta_0\frac{\ba^{*\intercal}\bX^{(0)}_\mu}{\sqrt{d_{0}}}+\sqrt{\gamma_0}\xi^{*}_\mu\Big)\,.
    \end{align*}
\end{corollary}

The estimate in this corollary holds for finite sizes of layers and training set. Furthermore, it identifies the scaling regime in which the two models become asymptotically equivalent, namely when the right-hand side of \eqref{boundMutualInfoAll} approaches zero as $n,(d_\ell)_{\ell=0}^L\to\infty$. We denote any such scaling limit by $\widetilde\lim$. Notice that this result shows that the trivialisation of deep neural networks in the Bayesian optimal setting happens even beyond the proportional regime $n=\Theta(d)$. Indeed, the r.h.s. of \eqref{boundMutualInfoAll} vanishes also if $n=o(d_m^{3/2})$. Notice also that from $I_n(\btheta^*,\bxi^*;\mathcal{D})$ it is possible to obtain the MI between $\btheta^*$ and $\mathcal{D}$ using the chain rule: $I_n(\btheta^*;\mathcal{D})=I_n(\btheta^*,\bxi^*;\mathcal{D})-I_n(\bxi^*;\mathcal{D}\mid \btheta^*)$. The limit of the last term in the latter is computed in Lemma~\ref{lem:relation_MIs}.

The mutual information between the GLM teacher network weights and the corresponding training dataset is a quantity that has been studied in the literature (\cite{JeanGLM_PNAS}), and its relevance lies in the fact
that the generalization error can be computed from it. As a consequence, we can prove the following:

\begin{theorem}[Generalization error equivalence]\label{thm:gen_error}
    Under the same hypothesis of Theorem \ref{thm:reduction} we have
    \begin{align}
        \widetilde \lim\,|\mathcal{E}^{(L)} - \mathcal{E}^{(0)}|=0\,,
    \end{align}
    where $\mathcal{E}^{(0)}$ is the GLM generalization error associated with $\frac{1}{n}{I}_n^{(0)}(\btheta^{*(0)},\bxi^*;\mathcal{D}_{0})$. 
\end{theorem}

It is important to note here that the deep neural network and the GLM are trained on two different datasets, each generated by the corresponding teacher network with a matching architecture. Specifically, $\mathcal{E}^{(0)}$ is not the generalization error attained by a GLM trained on $\mathcal{D}_L$.

The proof of Theorem 4 follows exactly that of Theorem 3 in \cite{camilli-tiepova-barbier2023fundamental} which is in turn leveraging on that of Theorem 4 in \cite{JeanGLM_PNAS}.

\section{Outline of the proofs}
\subsection{Proof of Theorem \ref{thm:reduction}}\label{subsec:proof_red}
The proof uses an interpolation argument, similar to those used in the rigorous theory of spin-glass \citep{interp_guerra_2002} and inference \citep{adaptive_original}. Since it does not change the derivation, we select $\rho=1,\epsilon=0$ for clarity. The interpolation is defined at the level of the inference problem itself. More specifically, define
\begin{align}\label{eq:interpol_S_s}
    S_{t\mu}&:=\sqrt{1-t}\frac{\mathbf{a}^{*\intercal}}{\sqrt{d_L}}\varphi\Big(\frac{\mathbf{W}^{*(L)}\mathbf{X}^{(L-1)}_\mu}{\sqrt{d_{L-1}}}\Big)+\sqrt{t}\rho_L\frac{\mathbf{v}^{*\intercal}\mathbf{X}^{(L-1)}_\mu}{\sqrt{d_{L-1}}}+\sqrt{t\epsilon_L}\zeta_\mu^{*(L)}\,,\\
    s_{t\mu}&:=\sqrt{1-t}\frac{\mathbf{a}^\intercal}{\sqrt{d_L}}\varphi\Big(\frac{\mathbf{W}^{(L)}\mathbf{x}^{(L-1)}_\mu}{\sqrt{d_{L-1}}}\Big)+\sqrt{t}\rho_L\frac{\mathbf{v}^{\intercal}\mathbf{x}^{(L-1)}_\mu}{\sqrt{d_{L-1}}}+\sqrt{t\epsilon_L}\zeta^{(L)}_\mu\,,
\end{align}
for any $t\in[0,1]$ and with independent Gaussian vector $\bv^*\sim \mathcal{N}(0,\mathbbm{1}_{d_L})$ and i.i.d. standard Gaussian noise variables $\zeta^{*(L)}_\mu,\zeta^{(L)}_\mu$, and $\bx_\mu^{(0)}=\bX_\mu^{(0)}$. Let the responses be drawn from $Y_{t\mu}\sim P_{  out} (\,\cdot\mid S_{t\mu})$. The parameters to infer in this model are $\btheta^{*(L)}=(\ba^*, (\bW^{*(\ell)}))$ and $\bv^*$. Let us introduce the convenient notation $u_y(x)=\log P_{  out}(y\mid x),u'_y(x)=\partial_x u_y(x)$, or sometimes just $u,u'$ for brevity. For this model we shall have an interpolating training set 
$\mathcal{D}_t=\{(Y_{t\mu},\bX^{(0)}_\mu)_{\mu=1}^n\}$, an interpolating Bayes posterior for $\btheta^{*(L)}\mid \mathcal{D}_t$, with expectation operator
\begin{align}\label{eq:int_posterior}
    \langle\,\cdot\,\rangle_t=\frac{1}{\cZ_t}dP(\btheta^{(L)}) \prod_{i=1}^{d_{L-1}}dP_\theta(v_i)\prod_{\mu=1}^n dP_\theta(\zeta^{(L)}_\mu)\exp\Big(u_{Y_{t\mu}}(s_{t\mu})\Big)(\cdot)
\end{align}
and an associated interpolating free entropy
\begin{align}\label{eq:Int_free_entropy}
    \bar f_{n,t}=\frac{1}{n}\EE_{(t)}\log\mathcal{Z}_t=\frac{1}{n}\EE_{(t)}\log\int dP(\btheta^{(L)})\EE_\bv\prod_{\mu=1}^n\EE_{\zeta_\mu^{(L)}}\exp\Big( u_{Y_{t\mu}}(s_{t\mu})\Big)\,.
\end{align}
For future convenience we also introduce the expectation over the law of the dataset $\mathcal{D}_t$:
\begin{align}
    \EE_{(t)}[\,\cdot\,]=\EE_{\btheta^{*(L)},\bv^*} \int \prod_{\mu=1}^n dY_{t\mu}\,\EE_{\bX_\mu^{(0)},\zeta^{*(L)}_\mu}\exp\Big( u_{Y_{t\mu}}(S_{t\mu})\Big)[\,\cdot\,]\,.
\end{align}
At $t=0$ the model is precisely the $L$-layer neural network, whereas at $t=1$ it reduces to the target $L-1$-layer neural network (recall $\rho=1,\epsilon=0$). Therefore, the goal is to control the derivative of the interpolating mutual information (obtained simply from the interpolating free entropy) w.r.t.\ $t$ and to prove it is vanishing. Here we shall focus only on the contribution to the difference on the l.h.s.\ of \eqref{eq:diff_MIs} given by the free entropies for $L$ and $L-1$ layers. The control of the first term on the r.h.s.\ in \eqref{eq:MI_definition} is deferred to the Appendix, see Lemma~\ref{lem:Psi_constant_MI}. The derivative of $\bar f_{n,t}$ decomposes as
\begin{align}\label{eq:der_free_entropy}
    \begin{split}
        \frac{d}{dt}\bar f_{n,t}=-A_1+A_2+A_3+B
    \end{split}
\end{align}
where, denoting $u_\mu^\prime:=\partial_x\log P_{out}(Y_{t\mu}\mid x)|_{x=S_{t\mu}}$  (not to be confused with $u^{\prime}_{Y_{t\mu}}(s_{t\mu})$) and similarly for higher derivatives $u_\mu'',u_\mu'''$, 
\begin{align}
    &A_1:=\frac{1}{2n\sqrt{(1-t)}}\E_{(t)}\log \cZ_t
    \sum_{\mu=1}^{n}u^{\prime}_\mu\frac{\mathbf{a}^{*\intercal}}{\sqrt{d_L}} \Big[
    \varphi\Big(\frac{\mathbf{W}^{*(L)}\mathbf{X}^{(L-1)}_\mu}{\sqrt{d_{L-1}}}\Big)-\rho_L\frac{\bW^{*(L)}\bX_\mu^{(L-1)}}{\sqrt{d_{L-1}}}\Big]\, ,\label{eq:A_1_def}\\
    &A_2:=\frac{1}{2n}\E_{(t)}\log \cZ_t
    \sum_{\mu=1}^{n}u^{\prime}_\mu\rho_L\Big[\frac{\mathbf{v}^{*\intercal}\mathbf{X}^{(L-1)}_\mu}{\sqrt{td_{L-1}}}-\frac{\ba^{*\intercal}\bW^{*(L)}\bX_\mu^{(L-1)}}{\sqrt{(1-t)d_Ld_{L-1}}}\Big]\, ,\label{eq:A_2_def}\\
    &A_3:=\frac{1}{2n}\E_{(t)}\log \cZ_t
    \sum_{\mu=1}^{n}u^{\prime}_\mu\sqrt{\frac{\epsilon_L}{t}}\zeta_\mu^{*(L)}\, ,\label{eq:A_3_def}\\
    \label{eq:B-def}
     &B:=\frac{1}{n}\E_{(t)}\Big\langle \sum_{\mu=1}^{n}u^{\prime}_{Y_{t\mu}}(s_{t\mu})\frac {ds_{t\mu}}{dt} \Big\rangle_t\,.
\end{align}

The statistical model associated with the posterior \eqref{eq:int_posterior} fulfills a set of symmetries called Nishimori identities, which are proved in Appendix \ref{app:nishiID}. These symmetries force $B$ to vanish identically; see Lemma \ref{lem:B_0}. Notice also that $A_1$ and $A_2$ share a common term which in one case is added, and in the other is subtracted. The reason for this operation is the following: in $A_2$ the two terms in the square parenthesis can be shown to have matching second moments, therefore they are supposed to compensate and cancel each other, which ends up being the case. The proof requires an integration by parts w.r.t.\ the Gaussian laws of $\bv^*$ and $\bW^{*(L)}$, and is deferred to Appendix~\ref{appx:details_th1}. Proving that $A_1$ and $A_3$ cancel each other requires more care. Let us explain the key ideas below.

Let us define $\balpha_\mu:=\bW^{*(L)}\bX_\mu^{(L-1)}/\sqrt{d_{L-1}}$ and $U_{\mu\nu}:=\delta_{\mu\nu}u''_\mu+u'_\mu u'_\nu$ for future convenience. We start from $A_1$ which is the most challenging. A natural way to simplify the expression is to integrate by parts the Gaussian readout weights $\ba^*$, obtaining
\begin{align}
    A_1=\frac{1}{2n}\EE_{(t)}\log\cZ_t\sum_{\mu,\nu=1}^nU_{\mu\nu}\frac{1}{d_L}\big[\varphi(\balpha_\mu)^\intercal\varphi(\balpha_\nu)-\rho_L\balpha_\mu^\intercal\varphi(\balpha_\nu)\big]\,.
\end{align}Notice that the sum over the data index $\mu$ has doubled since the same weights are interacting with all training sample. This will happen every time we integrate the network weights by parts. On the contrary, the variables $\zeta_\mu^{*(L)}$ are different for every training sample, hence an integration by parts of them in $A_3$ yields
\begin{align}
    A_3=\frac{\epsilon_L}{2n}\E_{(t)}\log \cZ_t
    \sum_{\mu=1}^{n}U_{\mu\mu}\,.
\end{align}Thanks to the fact that $\E[U_{\mu\nu}\mid \bX_\mu,\bX_\nu,\btheta^{*(L)}]=0$ (see Lemma \ref{lem:propertiesPout}), we can subtract from $\log\cZ_t$ its expectation without changing the value of the expressions. Therefore,
\begin{align}\label{eq:main_A3-A1}
    A_3-A_{1}&=\frac{1}{2}\EE_{(t)}\Big(\frac{1}{n}\log\cZ_t-\bar f_{n,t}\Big)\sum_{\mu=1}^nU_{\mu\mu}\Big[
    \epsilon_L-\frac{\|\varphi(\balpha_\mu)\|^2}{d_L}+\rho_L\frac{\balpha_\mu^\intercal\varphi(\balpha_\mu)}{d_L}
    \Big]\nonumber\\
    &-\frac{1}{2}\EE_{(t)}\Big(\frac{1}{n}\log\cZ_t-\bar f_{n,t}\Big)\sum_{\mu\neq \nu,1}^nU_{\mu\nu}\frac{1}{d_L}\big[\varphi(\balpha_\mu)^\intercal\varphi(\balpha_\nu)-\rho_L\balpha_\mu^\intercal\varphi(\balpha_\nu)\big]\,.
\end{align}
Let us focus on the first line. The strategy is to bound it using Cauchy-Schwartz, separating the difference between the random free entropy and its expectation from the summation over $\mu$. This yields a variance of the random free entropy that is $O(\sqrt{1/n + 1/d_m})$ by Theorem \ref{th:Z_k_concentr}. The other factor produced by Cauchy-Schwartz is an expectation of a double sum over $\mu,\nu=1,\dots,n$. For this, we can use the fact that $\E[U_{\mu\mu}\mid \bX_\mu,\btheta^{*(L)}]=0$ which restricts the sum to $\mu=\nu$. Furthermore, using that $\E[U^2_{\mu\mu}\mid \bX_\mu,\btheta^{*(L)}]$ is bounded by a constant (see Lemma \ref{lem:propertiesPout}), we are left with $n$ expectations of the square parenthesis squared. Here the choice of $\rho_L$ and $\epsilon_L$ becomes crucial, and it allows us to prove that the expectation of such square is $O(1/d_m)$ in Lemma~\ref{lem:A_3-A_1}.

The second line requires a more refined treatment. The summation over $\mu\neq \nu$ involves $O(n^2)$ centered random variables, which seems to go against our statement. What saves the day is that for $\mu\neq\nu$ the two terms in the square parenthesis are both tending to zero in the thermodynamic limit and in addition their leading orders are canceling each other, leaving us with a faster decay rate. 

In order to access these leading orders, we use a further \emph{circular interpolation}, that paves the way for a crucial integration by parts. More precisely, let $\tau\in[0,\pi/2]$ be a second interpolating parameter, and let $\tilde\bW^*$ be an independent copy of $\bW^{*(L)}$. Define then
\begin{align}
    \balpha_\mu(\tau):= \balpha_\mu \sin\tau+\frac{\tilde\bW^*\bX^{(L-1)}_\mu}{\sqrt{d_{L-1}}}\cos\tau\,.
\end{align} The great advantage of this interpolation is that $\dot{\balpha}_\mu(\tau):=\partial_\tau {\balpha}_\mu(\tau)$ is independent of $\balpha_\mu(\tau)$ in the measure $\EE_{\bW^{*(L)},\tilde\bW^*}$, because they are Gaussian (conditionally on $\bX_\mu^{(L-1)}$) and decorrelated. Using the fundamental theorem of integral calculus we can rewrite the second line of \eqref{eq:main_A3-A1} as
\begin{align}\label{eq:A_1_off_def}
A_1^{  off}:=\int_0^{\pi/2}d\tau\,\EE_{(t)}\Big(\frac{1}{n}\log\cZ_t-\bar f_{n,t}\Big)\sum_{\mu<\nu,1}^nU_{\mu\nu}\frac{1}{d_L}\dot{\balpha}_\mu(\tau)\circ\big[\varphi'(\balpha_\mu(\tau))-\rho_L\mathbf{1}\big]^\intercal\varphi(\balpha_\nu)\,,
\end{align}where $\mathbf{1}\in\mathbb{R}^{d_L}$ is the all-ones vector, and $\circ$ is the Hadamard entry-wise product. Notice that there is a term related to the $\tau=0$ point of the $\tau$-interpolation, that contains only $\varphi(\tilde\bW^*\bX_\mu^{(L-1)}/\sqrt{d_{L-1}})$. This factor is independent of everything else under $\EE_{\tilde\bW^*}$, and since $\varphi$ is odd it vanishes.

Now we finally have a Gaussian variable $\dot{\balpha}_\mu(\tau)$ (again, conditionally on $\bX_\mu^{(L-1)}$, i.e., w.r.t.\ the measure of the weights) to integrate by parts which, importantly, lies outside of the non-linearity. Another key point is that $\dot{\balpha}_\mu(\tau)$ is independent of $\balpha_\mu(\tau)$, but not of the other $\balpha_\nu$'s. Hence, the following correlations appear: $\EE_{\tilde\bW^*,\bW^{*(L)}}[\dot{\alpha}_{i\mu}(\tau)\alpha_{j\eta}]=\delta_{ij}
\bX_\mu^{(L-1)}\cdot\bX_\eta^{(L-1)}\cos\tau/{d_{L-1}}$. Defining $U_{\mu\nu\eta}:=u''_\mu u'_\nu\delta_{\mu\eta}+u'_\mu u''_\nu \delta_{\nu\eta} +u'_\mu u'_\nu u'_\eta $
an integration by parts w.r.t.\ $\dot{\balpha}_\mu(\tau)$ then yields
\begin{align}
    &A_1^{  off}=\int_0^{\pi/2}d\tau\,\cos\tau\,\EE_{(t)}\Big(\frac{1}{n}\log\cZ_t-\bar f_{n,t}\Big)\sum_{\mu<\nu,1}^n\sum_{\eta=1}^nU_{\mu\nu\eta}\frac{\bX_\mu^{(L-1)}\cdot\bX_\eta^{(L-1)} }{d_{L-1}} \label{eq:A_1off}\\
    &\qquad\qquad \qquad\qquad\times\sqrt{1-t}\sum_{i=1}^{d_L}
    \frac{a_i^*\varphi'(\alpha_{i\eta})}{\sqrt{d_L}} 
    \frac{\varphi'(\alpha_{i\mu}(\tau))-\rho_L}{d_L} \varphi(\alpha_{i\nu})\nonumber \\     &+\int_0^{\pi/2}d\tau\,\cos\tau\,\EE_{(t)}\Big(\frac{1}{n}\log\cZ_t-\bar f_{n,t}\Big)\sum_{\mu<\nu,1}^n U_{\mu\nu}\frac{\bX_\mu^{(L-1)}\cdot\bX_\nu^{(L-1)} }{d_{L-1}}\sum_{i=1}^{d_L}\frac{\varphi'(\alpha_{i\mu}(\tau))-\rho_L}{d_L}\varphi'(\alpha_{i\nu})\,.\nonumber
\end{align}
As we can see, now the scalar products between post-activations for different input data appear. This is crucial, as the post-activations preserve the ``quasi-orthogonality'' of the inputs, meaning $\EE(\bX^{(\ell)}_\mu\cdot\bX^{(\ell)}_\nu/ d_\ell)^{2k}=O(\min\{(d_{s})_{s=0}^\ell\}^{-k})$ whenever $\mu\neq\nu$, as proved in Lemma~\ref{lem:orthogonality propagation}.

Bearing in mind this fact, for both terms in \eqref{eq:A_1off} the proof proceeds with Cauchy-Schwartz inequality, separating the difference of free entropies from the summations over the training set indices $\mu,\nu$ and $\eta$. For both mentioned terms, the other factor produced by Cauchy-Schwartz inequality is a square of multiple summations of training set indices. In any case, said summations are simplified thanks to Lemma~\ref{lem:propertiesPout}. Specifically, using $\EE[U_{\mu\nu\eta}\mid \bX_\mu,\bX_\nu,\bX_\eta,\btheta^{*(L)}]=0=\EE[U_{\mu\nu}\mid \bX_\mu,\bX_\nu,\btheta^{*(L)}]$ and $\EE[U^2_{\mu\nu\eta}\mid \bX_\mu,\bX_\nu,\bX_\eta,\btheta^{*(L)}],\,\EE[U^2_{\mu\nu}\mid \bX_\mu,\bX_\nu,\btheta^{*(L)}]\leq C$ for some positive constant $C$. This reduces the number of indices we are summing over. With only these few ingredients, the first term in \eqref{eq:A_1off} is shown to be of the claimed order. Concerning the second term, one in addition needs to exploit that $(\sum_i\varphi'(\alpha_{i\mu}(\tau))/d_L-\rho_L)^2$ is concentrating to $0$ thanks to the definition of $\rho_L$. This concentration result requires a certain care and is dealt with in Lemma~\ref{lem:moments_post_activations}.



\subsection{Proof of Theorem~\ref{th:Z_k_concentr}}

In this section we discuss briefly the proof of Theorem~\ref{th:Z_k_concentr}. Even though it is formulated for the model \eqref{eq:output_channel_labels}-\eqref{eq:output_Gaussian_channel}, we consider the interpolated model \eqref{eq:interpol_S_s}-\eqref{eq:int_posterior} as the concentration of its free entropy is required in the proof of Theorem~\ref{thm:reduction}; it anyway recovers the initial model as $t=0$. To prove the concentration of the free entropy $f_{n,t}$  we adopt standard techniques such as Poincar\'e-Nash and Efron-Stein inequalities. For details, we refer to Appendix~\ref{app:concentr}.
We remind that $f_{n,t}$ is a function of the random parameters $(\bW^{*(\ell)})_{\ell=1}^L$, $(\bX^{(0)}_\mu,\zeta^{*(L)}_\mu,Z_\mu,\bA_\mu)_{\mu=1}^n$, $\ba^*$, $\bv^*$, and defined as
\begin{align}
    f_{n,t}=\frac{1}{n}\log\cZ_t=\frac{1}{n}\log\int dP(\btheta^{(L)})\EE_{\bv}\prod_\mu^n\EE_{\zeta_\mu^{(L)}} P_{out}(Y_{t\mu}|s_{t\mu}),
\end{align}
where $Y_{t\mu}=f(S_{t\mu};\bA_\mu)+\sqrt{\Delta}Z_\mu$. We will prove the concentration with respect to all random parameters in several steps. The first batch is $\psi_1=(\bX^{(0)}_\mu,\zeta^{*(L)}_\mu,Z_\mu)_{\mu=1}^n$, second is $\psi_2=(\bA_\mu)$, and lastly $\psi_3=((\bW^{*(\ell)})_{\ell=1}^L,\ba^*,\bv^*)$. For this, we use the variance decomposition 
\begin{align}\label{eq:var_free_ener}
     \mathbb{V}(f_{n,t})
    = \mathbb{E}\mathbb{V}_{\psi_1}(f_{n,t})
    +\mathbb{E}\mathbb{V}_{\psi_2}(\E_{\psi_1}f_{n,t})+
    \mathbb{V}_{\psi_3}(\mathbb E_{\psi_1,\psi_2}f_{n,t}).
\end{align}
Since all elements of set $\psi_1$ are independent Gaussian variables, we bound the first term by the Poincar\'e-Nash inequality (see Appendix~\ref{app:concentr}) as
\begin{align}\label{eq:der_psi_1}
        &\mathbb{E}\mathbb{V}_{\psi_1}\Big(\frac{1}{n}\log \cZ_t\Big)
        \leq\frac{1}{n^2}\mathbb{E}\sum_{\mu=1}^n\Big[\Big(\frac{\partial\log\cZ_t}{\partial\zeta^{*(L)}_\mu}\Big)^2
        +\Big(\frac{\partial\log\cZ_t}{\partial Z_\mu}\Big)^2
        +\sum_{i=1}^{d_{0}}\Big(\frac{\partial\log\cZ_t}{\partial X^{(0)}_{i\mu}}\Big)^2\Big].
    \end{align}
    In what follows $P_{  out}^{x}(y\mid x)$ denotes the derivative of the kernel w.r.t.\ $x$. A similar definition holds for $y$, and second derivatives. Also, for brevity we write $P_\mu=P_{out}(Y_{t\mu}|s_{t\mu})$.
    The derivatives with respect to  $\zeta^{*(L)}_\mu$ are easily bounded due to the boundedness of $f'$ and Lemma~\ref{lem:bound_P_out_der}:
    \begin{align*}
    \Big|\frac{\partial\log\cZ_t}{\partial\zeta^{*(L)}_\mu}\Big|\leq \Big|\Big\langle \frac{P_{\mu}^y}{P_\mu}\Big\rangle_t f^\prime(S_{t\mu};\bA_\mu)\sqrt{t\epsilon_L}\Big|\leq \sqrt{t\epsilon_L}C(f)(|Z_\mu|^2+1)\,.
\end{align*}
An analogous formula also holds for the derivatives w.r.t.\ $Z_\mu$. Since $\EE|Z_\mu|^k=O(1)$, we immediately obtain that the first two terms of $\eqref{eq:der_psi_1}$ are of order $n^{-1}$. 

In the last term, the elements $X^{(0)}_{i\mu}$ seem harder to access since they are hidden under the $L$ layers of non-linearity. In order to handle the corresponding derivative let us introduce, for any $1\leq \ell_1\leq \ell_2\leq L$ the $d_{\ell_2}\times d_{\ell_1-1}$ matrix
    \begin{align}
        \bB^{*[\ell_1:\ell_2]}_{\nu}:=\frac{\boldsymbol{\Phi}^{*(\ell_2-1)}_\nu\bW^{*(\ell_2)}\boldsymbol{\Phi}^{*(\ell_2-2)}_\nu\bW^{*(\ell_2-1)}\ldots\boldsymbol{\Phi}^{*(\ell_1-1)}_\nu\bW^{*(\ell_1)}}{\sqrt{d_{\ell_2-1}\ldots d_{\ell_1-1}}},
    \end{align}
    where $\boldsymbol{\Phi}^{*(\ell)}_\nu:=\text{diag}(\varphi^\prime(\bW^{*(\ell+1)}\bX_\nu^{(\ell)}/\sqrt{d_\ell}))$ is a $d_{\ell+1}\times d_{\ell+1}$ diagonal matrix. Then the derivative w.r.t.\ $X^{(0)}_{i\mu}$ becomes compact:
   \begin{align}
        \frac{\partial\log \cZ_t}{\partial X_{i\mu}^{(0)}}
        =\Big\langle\frac{P_{\mu}^y}{P_{\mu}}\Big\rangle f^\prime(S_{t,\mu};\bA_\mu)\Big(\sqrt{1-t}\frac{(\ba^{*\intercal}\bB_{\mu}^{*[1:L]})_i}{\sqrt{d_L}}+\sqrt{t}\rho_L\frac{(\bv^{*\intercal}\bB_{\mu}^{*[1:L-1]})_i}{\sqrt{d_{L-1}}}\Big)\nonumber\\
        + \sqrt{1-t}\Big\langle\frac{P_{\mu}^x}{P_{\mu}}\frac{(\ba^\intercal\bB_{\mu}^{[1:L]})_i}{\sqrt{d_L}}\Big\rangle
        +\sqrt{t}\rho_L\Big\langle\frac{P_{\mu}^x}{P_{\mu}}\frac{(\bv^\intercal\bB_{\mu}^{[1:L-1]})_i}{\sqrt{d_{L-1}}}\Big\rangle.
        \end{align}
  If all $d_\ell$ diverge with the same speed, the operator norm $d_{\ell-1}^{-1/2}\EE\|\bW^{(*\ell)}\|=O(1)$ is bounded, which leads to  $\EE\| \bB^{*[\ell_1:\ell_2]}_{\nu}\|=O(1)$ (recall $\varphi$ is Lipschitz). This, together with the Nishimori identities of Appendix~\ref{app:nishiID}, allows to bound the last term of \eqref{eq:der_psi_1} with $C(f,\varphi)n^{-1}$.
  
  The second term of \eqref{eq:var_free_ener} is bounded with the standard technique of taking a $pn$~dimensional random vector $\bA^\prime$, all elements of which are equal to those of $\bA:=\{\bA_\mu\}_\mu$ but one, which is an i.i.d. copy. Then we can evaluate the variance with respect to $\bA$ by bounding $\mathbb{E}_{\psi_1}f_{n,t}(\bA)-\mathbb{E}_{\psi_1}f_{n,t}(\bA^\prime)$ by means of Efron-Stein's inequality recalled in Appendix~\ref{app:concentr}.

The third term of \eqref{eq:var_free_ener} is the most challenging one. Denote $g:=\mathbb{E}_{\psi_1,\psi_2}f_{n,t}$. Thanks again to the Poincar\'e-Nash inequality we have
    \begin{align}
        \mathbb{E}(g-\EE_{\psi_3}g)^2\leq\sum_{\ell=1}^L\sum_{i_\ell,j_\ell}\EE\Big(\frac{\partial g}{\partial W_{i_\ell j_\ell}^{*(\ell)}}\Big)^2+\sum_{i}\EE\Big(\frac{\partial g}{\partial a_{i}^{*}}\Big)^2+\sum_{j}\EE\Big(\frac{\partial g}{\partial v_{j}^{*}}\Big)^2.
    \end{align}
The derivatives w.r.t. $a^*_i$ and $v^*_i$ are treated in the similar way as  the derivatives w.r.t. $W^{*(\ell)}_{i_{\ell} j_\ell}$, so we leave them to Appendix~\ref{app:concentr}. To give an idea of the proof, choose $\ell\in\{1,\ldots,L\}$ and write
\begin{align}\label{eq:der_W}
    \frac{\partial g}{\partial W_{ij}^{*(\ell)}}
    &=\EE_{\psi_1,\psi_2}\Big[\Big\langle\frac{P_{\mu}^y}{P_\mu}\Big\rangle f^\prime(S_{t\mu};\bA_\mu)\Big(\frac{\sqrt{1-t}(\ba^{*\intercal}\bB^{*[\ell+1:L]}_{\mu}\boldsymbol{\Phi}^{*(\ell-1)}_\mu)_{i}X^{(\ell-1)}_{j\mu}}{\sqrt{d_Ld_{\ell-1}}}\nonumber\\
    &\qquad\qquad+\frac{\sqrt{t}\rho_L(\bv^{*\intercal}\bB^{*[\ell+1:L-1]}_{\mu}\boldsymbol{\Phi}^{*(\ell-1)}_\mu)_{i}X^{(\ell-1)}_{j\mu}}{\sqrt{d_{L-1}d_{\ell-1}}}\Big)\Big].
\end{align}
We suppressed the $\mu$ summation with the normalization $n^{-1}$, since $(\psi_1, \psi_2)$ contains all the parameters that depend on $\mu$. Therefore the above expression is true for fixed arbitrary $\mu$. The two resulting terms are of similar form, so we focus only on the first one. A naive estimate yields only an $O(1)$, which is not enough. To reach the right order, we need another integration by parts. To do so we employ a similar trick as the one described in the previous section, i.e.,  a circular interpolation which replaces $\bX_\mu^{(0)}$ by an independent copy $\tilde{\bX}_\mu^{(0)}$ of it:
\begin{align*}
    \bX_\mu^{(0)}(\tau):=\bX^{(0)}_\mu\sin\tau+\tilde{\bX}^{(0)}_\mu\cos\tau\,,\quad \tau\in[0,\pi/2].
\end{align*}
Its derivative $\dot{\bX}_\mu^{(0)}(\tau)$ w.r.t. $\tau$ is independent of $\bX_\mu^{(0)}(\tau)$ and $\EE[X^{(0)}_{i\nu}\dot{X}^{(0)}_{j\mu}(\tau)]=\delta_{\mu\nu}\delta_{ij}\cos\tau$.
Respectively, in a recursive way we define $\bX_\mu^{(\ell)}(\tau):=\varphi(\bW^{*(\ell)}\bX^{(\ell-1)}_\mu(\tau)/\sqrt{d_{\ell-1})}$ and $\tilde{\bX}^{(\ell)}_\mu=\varphi(\bW^{*(\ell)}\tilde{\bX}^{(\ell-1)}_\mu/\sqrt{d_{\ell-1})}$ for $1\leq \ell\leq L$.  With this we can replace $\bX^{(\ell-1)}_{j\mu}$ in \eqref{eq:der_W} by 
\begin{align*}
    \varphi\Big(\frac{\bW_{j}^{*(\ell-1)}\tilde{\bX}^{(\ell-2)}_\mu}{\sqrt{d_{\ell-2}}}\Big)+\int_0^{\pi/2}d\tau \frac{d}{d\tau}\varphi\Big(\frac{\bW_{j}^{*(\ell-1)}\bX^{(\ell-2)}_\mu(\tau)}{\sqrt{d_{\ell-2}}}\Big)\,.
\end{align*}
The first term yields a vanishing contribution, since it is independent of everything else in \eqref{eq:der_W} and its expectation is 0. 
The second one will give  $ \sum_{p=1}^{d_0}B^{*[1:\ell-1]}_{jp}(\tau)\dot{X}_{p\mu}^{(0)}(\tau)$, which allows us  to integrate by parts w.r.t. $\dot{X}_{p\mu}^{(0)}(\tau)$. 
However, since in \eqref{eq:der_W} not only $\langle P^y_{out}/{P_{out}}\rangle$ and $S_{t\mu}$, but also each $\boldsymbol{\Phi}^{*(k)}_\mu$ (for $k=\ell-1,\ldots,L-1$) inside $\bB_\mu^{*[\ell+1:L-1]}$, depend on $\dot{X}_{p\mu}^{(0)}(\tau)$, the calculations become quite cumbersome. 
After integration by parts and some manipulations we end up with a finite number of terms, the majority of which can be bounded easily using the Nishimori identity, Lemma~\ref{lem:propertiesPout}, and a bound on the expected spectral norm $\EE\|\bB^{*[\ell_1:\ell_2]}\|$ (see Appendix~\ref{app:concentr}). However, some of the terms will have a more special structure, in particular 
\begin{align*}
    \EE\sum_{i,j}\Big|\frac{\partial g}{\partial W_{ij}^{*(\ell)}}\Big|^2\leq \frac{C(f)}{d_{\ell-1}}\EE\int_0^{\pi/2}\Tr\Big(\mathbf{E}^\intercal \text{diag}(\mathbf{e})\bG(\tau)\bG(\tau)^{\intercal} \text{diag}(\mathbf{e})\mathbf{E}\Big)d\tau+\ldots\,.
\end{align*}
Here, matrices $\mathbf{E}$ and $\bG(\tau)$   are products of several matrices of the type $\bW^{*(k)}/\sqrt{d_{k-1}}$, $\boldsymbol{\Phi}_\mu^{*(k)}(\tau)$ and $\boldsymbol{\Phi}_\mu^{*(k)}$, while $\text{diag}(\mathbf{e})$ is a diagonal matrix with vector $\mathbf{e}=\ba^{*\intercal}\bB_\mu^{*[k_1:k_2]}$ on the diagonal (for some $k_1,k_2$). As already noted, the operator norms of the mentioned matrices are all order $O(1)$, which makes the operator norms of $\mathbf{E} $ and $\bG(\tau)$ also bounded. 
This allows to evaluate such term using the fact that for a symmetric semidefinite positive matrix $\bA$ and any matrix $\bB$ we have $|\Tr(\bA\bB)|\leq \|\bB\|\Tr \bA$:
\begin{align*}
      &\frac{C(f)}{d_{\ell-1}}\EE\int_0^{\pi/2}\Tr\Big(\mathbf{E}^\intercal \text{diag}(\mathbf{e})\mathbf
      G(\tau)\mathbf{G}(\tau)^\intercal \text{diag}(\mathbf{e})\mathbf{E}\Big)d\tau
      \\
      &\qquad\leq \frac{C(f)}{d_{\ell-1}}\EE\int_0^{\pi/2}\|\mathbf{E}\|^2\|\mathbf{G}(\tau)\|^2\Tr \,\text{diag}(\mathbf{e})^2d\tau\\
      &\qquad=\frac{C(f)}{d_{\ell-1}}\EE\int_0^{\pi/2}\|\mathbf{E}\|^2\|\mathbf{G}(\tau)\|^2\|\mathbf{e}\|^2d\tau\le \frac{C(f,\varphi)}{d_{\ell-1}}\,.
  \end{align*}
  Hence, as long as all $d_\ell$ are proportional we have $\mathbb{V}(f_{n,t})\leq C(f,\varphi)(\frac{1}{n}+\frac{1}{d_m})$.

\section{Conclusion and perspectives}
We have analysed deep neural networks in the Bayes-optimal teacher-student setting. Our main result proves that in the proportional scaling regime where the input dimension $d_0$, width $(d_{\ell})$ and number of training samples $n$ all grow large at fixed ratios, from the information-theoretic viewpoint, hidden layers can be iteratively replaced by properly defined effective \emph{linear} layers. Repeating this process until the exhaustion of all inner layers, the outcome is a fundamental equivalence between deep neural networks and generalised \emph{linear} models in this setting. As a consequence, the known rigorous results on generalised linear models directly give explicit formulas for the main information-theoretic quantities of interest: mutual information and optimal generalisation error.

The main conceptual take-home message is that in order to escape this equivalence with linear models, deep neural networks must be analysed beyond the proportional regime. This means that the number of samples $n$ must grow much faster than $d_0,(d_{\ell})$. Counting the number of trainable parameters suggests that the non-trivial scaling corresponds to $n=\Theta({\rm max}\{d_{\ell}\}^2)=\Theta(d_0^2)$ when all widths are proportional and the number of layers is fixed. This regime is of a fundamentally different nature and comes with numerous technical challenges. Indeed, it has been resisting approaches based on statistical physics up to very recently, see \cite{maillard2024bayes} for the analysis in the special case of quadratic inner activation function with Gaussian weights, and \cite{barbier2025optimal} for a generic framework able to tackle any activation function and distribution of weights (which is closely related to \cite{barbier2024phase}). This beyond-proportional scaling regime is, from the authors' perspective, an important one to focus on in order to make progress in bridging the gap between theory and practice of modern neural networks.

\section*{Acknowledgements}
J.B., F.C., D.T were funded by the European Union (ERC, CHORAL, project number 101039794). Views and opinions expressed are however those of the authors only and do not necessarily reflect those of the European Union or the European Research Council. Neither the European Union nor the granting authority can be held responsible for them.

\vskip 0.2in
\bibliography{colt-main2025}

\appendix

\section{Nishimori identity}\label{app:nishiID}
The Nishimori identities are a very general set of symmetries arising in inference in the Bayes-optimal setting as a consequence of Bayes' rule. They were initially discovered in the context of the gauge theory of spin glasses \cite{nishimori01}, which possess a sub-region of their phase space, called \emph{Nishimori line}, where the most relevant thermodynamic quantities can be exactly computed, and we can generally count on replica symmetry, namely the concentration of order parameters \cite{barbier2022strong}.

To introduce them, consider a generic inference problem where a Bayes-optimal statistician observes $\bY$ that is a random function of some ground truth signal $\bX^*$: $\bY\sim P_{Y|X}(\bX^*)$. Then the following holds:
\begin{proposition}[Nishimori identity]
For any bounded function $f$ of the signal $\bX^*$, the data $\bY$ and of conditionally i.i.d. samples from the posterior $\bx^j\sim P_{X\mid Y}(\,\cdot \mid \bY)$, $j=1,2,\ldots,n$, we have that
\begin{align}
    \EE\langle f(\bY,\bX^*,\bx^2,\ldots,\bx^{n})\rangle=\EE\langle f(\bY,\bx^1,\bx^2,\ldots,\bx^{n})\rangle
\end{align}
where the bracket notation $\langle \,\cdot\,\rangle$ is used for the joint expectation over the posterior samples $(\bx^j)_{j\le n}$, $\EE$ is over the signal $\bX^*$ and data $\bY$.
\end{proposition}

\begin{proof}
    The proof based on Bayes' rule is elementary, see, e.g., \cite{Lelarge2017FundamentalLO}.
\end{proof}

\section{Lipschitz functions of sub-Gaussian random variables}\label{app:Lemmas}
In this section we collect Lemmas for the proof of Theorem \ref{thm:reduction}. The first one hereby illustrates how to use the notion of orthogonality, i.e. independence, between Gaussians to expand nonlinear functions of them. This will be needed to handle the activation functions.

\begin{lemma}[Orthogonal approximation]\label{lem:approx}
Let $\psi\in C^2(\mathbb{R})$ be an odd function with bounded first and second derivatives, and $g_1,g_2\in C^1(\mathbb{R})$ where $g_1$ is a Lipschitz function, and $g_2,g_2'$ have at most polynomial growth. Let $z_1$ and $z_2$ be two centered Gaussian random variables with covariance $\EE z_i z_j=C_{ij}$, for $i,j=1,2$. Let $\Delta_{12}=C_{12}/C_{22}$. Then, for some constants $L_1,L_2>0$
\begin{align}
\label{eq:approx1}
    &|\E  \psi(z_1)g_2(z_2)-C_{12}\E \psi'(z_1)\E g_2'(z_2)|\leq L_1\Delta_{12}^2,\\
\label{eq:approx2}
    &|\E  g_1(z_1)g_2(z_2)-\E g_1(z_1)\E g_2(z_2)|\leq L_2\Delta_{12}.
\end{align}
\end{lemma}

\begin{proof}
Define $z_{1\perp 2}=z_1-z_2C_{12}/C_{22}=z_1-\Delta_{12}z_2$, which rewrites as $z_1=z_{1\perp2}+\Delta_{12}z_{2}$. This is nothing but Grahm-Schmidt orthogonalization procedure, used to make $z_{1\perp 2}$ independent on $z_2$. Using a Taylor expansion with integral remainder we have
\begin{align*}
    \E \psi(z_{1})g_2(z_2)-\Delta_{12}\EE  \psi'(z_{1\perp 2})\EE z_2g_2(z_2)=\Delta_{12}^2\int_0^1(1-s)\EE \psi''(z_{1\perp2}+s\Delta_{12}z_2)z_2^2 g_2(z_2)  ds
\end{align*}where we have used oddity of $\psi$ to eliminate the zero-th order of the expansion. Now, using the fact that $\psi$ has second derivative uniformly bounded by a constant $K$, we get
\begin{align}
    |\EE \psi''(z_{1\perp2}+s\Delta_{12}z_2)z_2^2 g_2(z_2)|\leq K\EE z_2^2 |g_2(z_2)|\leq\bar{K}
\end{align}thanks to the at most polynomial growth of $g_2$, for a proper constant $\bar{K}$. Then we reuse the same argument for $\EE  \psi'(z_{1\perp 2})$:
\begin{align}
    |\EE  \psi'(z_{1\perp 2})-\EE  \psi'(z_{1})|\leq
    \Delta_{12}\int_0^1 \EE |\psi''(z_1-sz_2\Delta_{12})z_{2}| ds\leq \tilde K \Delta_{12}\,,
\end{align}
where we use that the boundedness of the derivatives of $\psi$, for a proper constant $\tilde K$. Together with the integration by parts $\EE z_2g_2(z_2)=C_{22}\EE g'_2(z_2)$, this leads directly to \eqref{eq:approx1}.

For \eqref{eq:approx2} instead, we have
\begin{align}
    |\E  g_1(z_1)g_2(z_2)-\E  g_1(z_{1\perp2})\E g_2(z_2)|\leq\Delta_{12}\int_0^1\EE |g_1'(z_{1\perp2}+s\Delta_{12}z_2) z_2 g_2(z_2)|\,ds\,.
\end{align} In order to bound the remainder it suffices to use the fact that $g_1$ is Lipschitz, and the at most polynomial growth of $g_2,g_2'$. Then, to replace $z_{1\perp2}$ by $z_{1}$ in $g_1$ one uses the same strategy backward, producing the same order of remainder.
\end{proof}

Following this, we need to make sure that the \enquote{quasi-orthogonality} condition of the inputs, namely the exponential concentration property
\begin{align}
\mathbb{P}\Big(\Big|\frac{\bX^{(0)}_\mu\cdot\bX^{(0)}_\nu}{d_0}-\delta_{\mu\nu}\Big|\geq\epsilon\Big)\leq C\exp\big(-c\epsilon^2\big)
\end{align}for any $\epsilon>0$ and some $C,c>0$, propagates to all the post activations in some form.
This is the aim of the next Lemmas, which are stated for a general setting.

\begin{lemma}[Concentration propagation]\label{lem:concentration_propagation}
    Consider two deterministic vectors $\bX,\bY\in\mathbb{R}^d$ with Euclidean norms $\sigma_X$ and $\sigma_Y$ respectively, and $\bW\in\mathbb{R}^{p\times d}$ a random matrix with i.i.d.\ centred sub-Gaussian entries. Consider two Lipschitz functions $\phi,\psi\in C^1(\mathbb{R})$, applied element-wise to vectors. Then, for all $\epsilon>0$
    \begin{align}
        \mathbb{P}\Big(\Big|\frac{\phi(\bW\bX)^\intercal\psi(\bW\bY)}{p}-
        \EE \frac{\phi(\bW\bX)^\intercal\psi(\bW\bY)}{p}\Big|\geq \epsilon\Big)\leq2\exp\Big[-
        \frac{p\epsilon^2}{2L\sigma_X\sigma_Y(\epsilon+L\sigma_X\sigma_Y)}
        \Big]
    \end{align}for some $L>0$ depending on $\phi,\psi$.
\end{lemma}
\begin{proof}
    Define for brevity $\alpha=\bW_1\cdot\bX$, $\beta=\bW_1\cdot\bY$. We start by considering one side of the bound, as the other one follows from the same arguments. Using exponential Markov's inequality we get
    \begin{align}
        \mathbb{P}
        \Big(\frac{\phi(\bW\bX)^\intercal\psi(\bW\bY)}{p}-
        \EE \frac{\phi(\bW\bX)^\intercal\psi(\bW\bY)}{p}
        \geq \epsilon\Big)\leq e^{-s(\epsilon+\EE\phi(\alpha)\psi(\beta))}\Big[\EE e^{s\frac{\phi(\alpha)\psi(\beta)}{p}}\Big]^p
    \end{align}for any $s>0$.
    Let us introduce the generating function
    \begin{align}
        g_p(s):=&\,p\log\EE e^{s\frac{\phi(\alpha)\psi(\beta)}{p}}\leq p\Big(\EE e^{s\frac{\phi(\alpha)\psi(\beta)}{p}}-1\Big)\nonumber\\
        =&\,s\EE\phi(\alpha)\psi(\beta)
        +p\sum_{k=2}^\infty
        \frac{s^k}{k!p^k}\EE \phi(\alpha)^k\psi(\beta)^k\,.
    \end{align}Now we can use the fact that $\phi,\psi$ are Lipschitz functions applied to sub-Gaussian random variables. Hence $\phi(\alpha),\psi(\beta)$ are sub-Gaussian random variables. As a consequence, there exists a $L$ such that
    \begin{align}
        g_p(s)-s\EE\phi(\alpha)\psi(\beta)\leq \sum_{k=2}^\infty
        \frac{s^k(L\sigma_X\sigma_Y)^k}{2p^{k-1}}=\frac{(s L\sigma_X\sigma_Y)^2}{2p}\frac{1}{1-\frac{s L\sigma_X\sigma_Y}{p}}\,,
    \end{align}which restricts us to choose $s<p/(L \sigma_X\sigma_Y)$. Therefore, our bound on the probability becomes
    \begin{align}
        \mathbb{P}
        \Big(\frac{\phi(\bW\bX)^\intercal\psi(\bW\bY)}{p}-
        \EE \frac{\phi(\bW\bX)^\intercal\psi(\bW\bY)}{p}
        \geq \epsilon\Big)\leq e^{-s\epsilon+\frac{(s L \sigma_X\sigma_Y)^2}{2p}\frac{1}{1-\frac{s L \sigma_X\sigma_Y}{p}}}\,.
    \end{align}
    The result then follows from the choice $s=p\epsilon/(L \sigma_X\sigma_Y(\epsilon+L \sigma_X\sigma_Y))$.
\end{proof}

\begin{lemma}[Moments bounds]\label{lem:moments_bounds}
    Let $k$ be a positive integer, and $\bW,\phi,\psi,\bX,\bY$ as in Lemma \ref{lem:concentration_propagation}. Then 
    \begin{align}\label{eq:moment_control_1}
        \EE\Big|\frac{\phi(\bW\bX)^\intercal\psi(\bW\bY)}{p}-\EE\frac{\phi(\bW\bX)^\intercal\psi(\bW\bY)}{p}\Big|^k\leq \frac{C_{k}(\sigma_X\sigma_Y)^k}{p^{k/2}}\,,
    \end{align}for a positive constant $C_{k}>0$ depending on $\phi,\psi$. Furthermore, if $\psi\in C^2(\mathbb{R})$ is an odd function with bounded first and second derivatives, then
    \begin{align}\label{eq:moment_control_2}
        \EE\Big|\frac{\phi(\bW\bX)^\intercal\psi(\bW\bY)}{p}\Big|^k\leq \frac{C'_{k}(\sigma_X\sigma_Y)^k}{p^{k/2}}+L_{k}|\bX^\intercal\bY|^k+L'_{k}\Big(\frac{|\bX^\intercal\bY|}{\|\bY\|^{2}}\Big)^{2k}
    \end{align}for some positive constants $C'_{k},L_{k},L'_{k}>0$ depending on $\phi,\psi$. 
\end{lemma}
\begin{proof} The first inequality follows from the tail integration. Define the shortcut notation $\phi_\bX=\phi(\bW\bX)$, and similarly for $\psi_\bX$. Then
\begin{align}
    \EE\Big|\frac{\phi_\bX^\intercal\psi_\bY}{p}-\EE\frac{\phi_\bX^\intercal\psi_\bY}{p}\Big|^k&=k\int_0^\infty dt\,t^{k-1}\mathbb{P}\Big(\Big|\frac{\phi_\bX^\intercal\psi_\bY}{p}-\EE\frac{\phi_\bX^\intercal\psi_\bY}{p}\Big|\geq t\Big)\nonumber\\
    &\leq 2k\int_0^\infty dt\,t^{k-1}e^{-\frac{pt^2}{2L\sigma_X\sigma_Y(t+L\sigma_X\sigma_Y)}}.
\end{align}
We split the integral into two contributions and we bound them separately. The first one is
\begin{align}
    \int_0^{L\sigma_X\sigma_Y} dt\,t^{k-1}e^{-\frac{pt^2}{2L\sigma_X\sigma_Y(t+L\sigma_X\sigma_Y)}}\leq
    \int_0^{L\sigma_X\sigma_Y} dt\,t^{k-1}e^{-\frac{pt^2}{4(L\sigma_X\sigma_Y)^2}}.
\end{align}It suffices then to extend the integral to $\infty$ (the integrand is positive), and to perform the change  of variables $t\mapsto t L\sigma_X\sigma_Y/\sqrt{p}$ to show that
\begin{align}
    \int_0^{L\sigma_X\sigma_Y} dt\,t^{k-1}e^{-\frac{pt^2}{2L\sigma_X\sigma_Y(t+L\sigma_X\sigma_Y)}}\leq \tilde C_{k}\frac{(\sigma_X\sigma_Y)^k}{p^{k/2}}\,.
\end{align}
The second piece of the integral is instead
\begin{align}
    \int_{L\sigma_X\sigma_Y}^{\infty} dt\,t^{k-1}e^{-\frac{pt^2}{2L\sigma_X\sigma_Y(t+L\sigma_X\sigma_Y)}}\leq
    \int_{L\sigma_X\sigma_Y}^{\infty} dt\,t^{k-1}e^{-\frac{pt}{4L\sigma_X\sigma_Y}}\leq \hat{C}_{k}\frac{(\sigma_X\sigma_Y)^k}{p^k}.
\end{align}One then just sets $C_{k}=2k(\tilde C_{k}+\hat C_{k})$ and \eqref{eq:moment_control_1} is proved.

Equation \eqref{eq:moment_control_2} can be derived from \eqref{eq:moment_control_1} and Lemma \ref{lem:approx} as follows. Firstly, by triangular and Jensen's inequality respectively  we get:
\begin{align}
    \EE\Big|\frac{\phi_\bX^\intercal\psi_\bY}{p}\Big|^k&\leq \EE\Big(\Big|\frac{\phi_\bX^\intercal\psi_\bY}{p}-\EE\frac{\phi_\bX^\intercal\psi_\bY}{p}\Big|+\Big|\EE\frac{\phi_\bX^\intercal\psi_\bY}{p}\Big|\Big)^k\nonumber\\
    &\leq 2^{k-1}\EE\Big|\frac{\phi_\bX^\intercal\psi_\bY}{p}-\EE\frac{\phi_\bX^\intercal\psi_\bY}{p}\Big|^k+2^{k-1}\Big|\EE\frac{\phi_\bX^\intercal\psi_\bY}{p}\Big|^k\,.
\end{align}
Secondly, the first term on the r.h.s.\ of the above can be bounded using \eqref{eq:moment_control_1}. Thirdly, Lemma \ref{lem:approx} on $\EE\phi_\bX^\intercal\psi_\bY/p$ readily yields the statement.
\end{proof}

\section{Bounds for post-activations moments}
Note that Lemmas \ref{lem:concentration_propagation} and \ref{lem:moments_bounds} hold for any couple of vectors $\bX,\bY$. In particular, \eqref{eq:moment_control_2} becomes pointless when $\bX=\bY$, because the bound does not vanish. On the contrary, in that case equation \eqref{eq:moment_control_1} gives us the concentration of the norms of the post activations $\varphi(\bW\bX)$ when choosing $\phi=\psi=\varphi$. When $\bX$ and $\bY$ are instead quasi-orthogonal, \eqref{eq:moment_control_2} is the reflection in terms of moments of the propagation of such quasi-orthogonality property to the post-activations $\varphi(\bW\bX),\varphi(\bW\bY)$.

As a consequence of these Lemmas one has  the desired moment control for all post-activations:
\begin{lemma}[Post-activations moments]\label{lem:moments_post_activations}
 Define the sequence of coefficients:
    \begin{align}
         \sigma_0=1\,,\,\sigma_\ell=\EE\varphi^2(Z\sqrt{\sigma_{\ell-1}})\,,\quad \rho_\ell=\EE\varphi'(Z\sqrt{
        \sigma_{\ell-1}})
    \end{align}where $\ell\in[L]$ and $Z\sim\mathcal{N}(0,1)$. 
    Let $k$ be a positive integer and $g(x)=\varphi^2(x),\varphi'(x)$ or $x\varphi(x)$. Then 
    \begin{align}\label{eq:coefficients_convergence_rate}
        \EE\Big|\frac{1}{d_\ell}\sum_{i=1}^{d_\ell}g\Big(\frac{\bW^{*(\ell)}_i\cdot\bX^{(\ell-1)}_\mu}{{\sqrt{d_{\ell-1}}}}\Big) -\EE g(Z\sqrt{\sigma_{\ell-1}})\Big|^k\leq \frac{C_{k,\varphi}}{d_{[0:\ell]}^{k/2}}\,,
    \end{align}for a positive constant $C_{k,\varphi}>0$, where $d_{[0:\ell]}=\min\{(d_s)_{s=0}^\ell\}$.
\end{lemma}
\begin{proof}
The key proof is the one for $g=\varphi^2$. In that case the statement reads:
\begin{align}\label{eq:intermediate_momentcontrol}
    \EE\Big|\frac{\|\bX_\mu^{(\ell)}\|^2}{d_\ell}-\sigma_\ell\Big|^k\leq \frac{C_{k,\varphi}}{d_{[0:\ell]}^{k/2}}\,.
\end{align}
    We proceed by induction. For $\ell=0$ the statement is surely true. Now assume it holds for $\ell-1$. The l.h.s.\ of \eqref{eq:intermediate_momentcontrol} can be bounded via triangular and Jensen inequalities as
    \begin{align}
        ...\leq 2^{k-1}\EE\Big|\frac{\|\bX_\mu^{(\ell)}\|^2}{d_\ell}&-\EE_Z\varphi^2\Big(Z\sqrt{\frac{\|\bX^{(\ell-1)}_\mu\|^2}{d_{\ell-1}}}\Big)\Big|^k\nonumber\\
        &+2^{k-1}\EE\Big|
        \EE_Z\varphi^2\Big(Z\sqrt{\frac{\|\bX^{(\ell-1)}_\mu\|^2}{d_{\ell-1}}}\Big)-\EE_Z\varphi^2(Z\sqrt{\sigma_{\ell-1}})
        \Big|^k\,.
    \end{align}
    The first of the two terms can be bounded simply using \eqref{eq:moment_control_1} for $\bX=\bY$, $\phi=\psi=\varphi$. Let us focus on the second. Define
    \begin{align}
        R:=
        \EE_Z g\Big(Z\sqrt{\frac{\|\bX^{(\ell-1)}_\mu\|^2}{d_{\ell-1}}}\Big)-\EE_Z g(Z\sqrt{\sigma_{\ell-1}})
        =\EE_Z\int_0^1 ds g'(Z(s))Z'(s)
    \end{align}where $$Z(s):=Z\Big(s\sqrt{\frac{\|\bX^{(\ell-1)}\|^2}{d_{\ell-1}}}+(1-s)\sqrt{\sigma_{\ell-1}}\Big).$$ Since $\varphi,\varphi'$ are both Lipschitz 
    there exists a constant $\bar{K}_\varphi$ such that $|g'(z)|\leq\bar K_\varphi|z|$. Hence
    \begin{align}
        |R|&\leq \bar K_\varphi\Big|\sqrt{\frac{\|\bX_\mu^{(\ell-1)}\|^2}{d_{\ell-1}}}-\sqrt{\sigma_{\ell-1}}\Big|\,{ \EE_Z|Z|^2}
        \int_0^1 ds\, \Big(s\sqrt{\frac{\|\bX^{(\ell-1)}\|^2}{d_{\ell-1}}}+(1-s)\sqrt{\sigma_{\ell-1}}\Big)\nonumber\\
        &=
        \frac12\bar K_\varphi\Big|{\frac{\|\bX_\mu^{(\ell-1)}\|^2}{d_{\ell-1}}}-{\sigma_{\ell-1}}\Big|\,,
    \end{align} 
    and therefore $\EE|R|^k\leq C'_{\varphi,k} d_{[0:\ell-1]}^{-k/2}$ for a positive constant $C'_{\varphi,k}$ by the inductive hypothesis.

    Given that the statement is true for $g=\varphi^2$, proving it for $g=\varphi'$ follows exactly the same steps, using \eqref{eq:moment_control_1} with $\phi=\varphi'$ and $\psi=1$. Analogously for $g(x)=x\varphi(x)$, \eqref{eq:moment_control_1} can be used with $\phi=\varphi$ and $\psi= Id$ {(recall that the first and second derivative of $\varphi$ are bounded).}
\end{proof}

Finally, the quasi orthogonality of the inputs, i.e.
\begin{align}
    \EE\Big|\frac{\bX_\mu^{(0)}\cdot \bX_\nu^{(0)}}{d_0}\Big|^k\leq\frac{C_k}{d_0^{k/2}}
\end{align}is shown to propagate to subsequent layers:
\begin{lemma}[Quasi-orthogonality propagation]\label{lem:orthogonality propagation}
    For all $\ell=0,\dots,L$, and a positive integer $k$ there exists positive a constant $C_k$ such that
    \begin{align}
        \EE\Big|\frac{\bX_\mu^{(\ell)}\cdot \bX_\nu^{(\ell)}}{d_\ell}\Big|^k\leq\frac{C_k}{ d_{[0:\ell]}^{k/2}}\,,
    \end{align}for all $\mu\neq \nu$, where $d_{[0:\ell]}=\min\{(d_s)_{s=0}^\ell\}$.
\end{lemma}
\begin{proof}
    The proof proceeds by induction. The statement is certainly true for $\ell=0$. 
    
    Assume now it holds for $\ell-1$, and let us rewrite the l.h.s.\ of the inequality as
    \begin{align}\label{eq:intermediate_111}
        \EE\EE_{\bW^{*(\ell)}}\Big|\frac{1}{d_\ell} \varphi\Big(\frac{\bW^{*(\ell)}\bX_\mu^{(\ell-1)}}{\sqrt{d_{\ell-1}}}\Big)\cdot \varphi\Big(\frac{\bW^{*(\ell)}\bX_\nu^{(\ell-1)}}{\sqrt{d_{\ell-1}}}\Big)\Big|^k\,.
    \end{align}Using the fact that $\varphi$ is Lipschitz, and Cauchy-Schwartz inequality, it is clear that the above object is bounded by a constant. Furthermore, from Lemma~\ref{lem:moments_post_activations} and Markov's inequality we know that
    \begin{align}\label{eq:Markov_moments_post_activ}
        \mathbb{P}\Big(\Big|\frac
        {\|\bX_\mu^{(\ell-1)}\|^2}{d_{\ell-1}}-\sigma_{\ell-1}\Big|\geq\delta\Big)\leq \frac{1}{\delta^{2k}}\EE
        \Big|\frac
        {\|\bX_\mu^{(\ell-1)}\|^2}{d_{\ell-1}}-\sigma_{\ell-1}\Big|^{2k}\leq\frac{C_{2k,\varphi}}{d_{[0:\ell-1]}^{k}\delta^{2k}}
    \end{align}for any $\delta>0$. If we call the event measured above $B_{\mu,\delta}$, then there exist a positive constant $C_k'$ such that
    \begin{align}
        \EE\EE_{\bW^{*(\ell)}}\Big|\frac{1}{d_\ell} \varphi\Big(\frac{\bW^{*(\ell)}\bX_\mu^{(\ell-1)}}{\sqrt{d_{\ell-1}}}\Big)\cdot \varphi\Big(\frac{\bW^{*(\ell)}\bX_\nu^{(\ell-1)}}{\sqrt{d_{\ell-1}}}\Big)\Big|^k\mathbbm{1}(B_{\mu,\delta})\mathbbm{1}(B_{\nu,\delta})\leq\frac{C_{k}'}{d_{[0:\ell-1]}^{k/2}\delta^k}\,.
    \end{align}Hence, up to an error we can afford, we can always use the fact that $d_{\ell-1}(\sigma_{\ell-1}-\delta)\leq \|\bX_\mu^{(\ell-1)}\|^2\leq d_{\ell-1}(\sigma_{\ell-1}+\delta)$. With this observation, using \eqref{eq:moment_control_2} we can bound \eqref{eq:intermediate_111} with
    \begin{align}
        ...\leq \frac{C_k'}{d_\ell^{k/2}}(\sigma_{\ell-1}+\delta)^{2k}
        +L_k\EE\Big|\frac{\bX_\mu^{(\ell-1)}\cdot\bX_\nu^{(\ell-1)}}{d_{\ell-1}}\Big|^k+L_k'\EE\Big(\frac{\bX_\mu^{(\ell-1)}\cdot\bX_\nu^{(\ell-1)}}{d_{\ell-1}(\sigma_{\ell-1}-\delta)}\Big)^{2k}.
    \end{align}
    To conclude the proof it suffices to choose $\delta=\sigma_{\ell-1}/2$, and to use the inductive hypothesis.
\end{proof}

\section{Properties of the output kernel}
Recall the definition \eqref{eq:output_kernel_vs_Gauss_channel} of the output kernel $P_{  out}$. Based on this we are able to prove certain properties about $u_{Y_{t\mu}}(S_{t\mu})$ and its derivatives that are inherited from those of the readout function $f$.
For this section $P_{  out}^{x}(y\mid x)$ denotes the derivative of the kernel w.r.t.\ $x$. A similar definition holds for $y$, and second derivatives.

We start with a computation of the derivatives of $P_{  out}$:

\begin{lemma}\label{lem:bound_P_out_der}
Let $y=Y_{t\mu}=f(S_{t\mu};\bA_\mu)+\sqrt{\Delta}Z_\mu$. Let $f(\cdot;\bA)$ be a $C^2(\mathbb{R})$ with bounded first and second derivatives. Then there exists constant $C(f)$ such that
\begin{align}
    \max\Big\{\Big|\dfrac{P_{  out}^y(y|x)}{P_{  out}(y|x)}\Big|,
    \Big|\dfrac{P_{  out}^x(y|x)}{P_{  out}(y|x)}\Big|,
    \Big|\dfrac{P_{  out}^{yy}(y|x)}{P_{  out}(y|x)}\Big|,
    \Big|\dfrac{P_{  out}^{yx}(y|x)}{P_{  out}(y|x)}\Big|,
    \Big|\dfrac{P_{  out}^{xx}(y|x)}{P_{  out}(y|x)}\Big|\Big\}<C(f)(|Z_\mu|^2+1)\,.
\end{align}
\end{lemma}
\begin{proof}
    For the proof we refer to \cite[Lemma 17]{camilli-tiepova-barbier2023fundamental}.
\end{proof}

\begin{lemma}[Properties of $P_{  out}$]\label{lem:propertiesPout}
Recall the definition $u_y(x):=\log P_{  out}(y\mid x)$. We denote $u'_y\equiv u'_y(x):=\partial_x u_y(x)$. Furthermore, let 
\begin{align}
    \label{eq:U_munu}
    U_{\mu\nu}&:=\delta_{\mu\nu} u''_{Y_{t\mu}}(S_{t\mu})+u'_{Y_{t\mu}}(S_{t\mu}) u'_{Y_{t\nu}}(S_{t\nu})\,,\\
    \label{eq:U_munueta}
    U_{\mu\nu\eta}&:=u'''_{Y_{t\mu}}(S_{t\mu})\delta_{\mu\nu} \delta_{\nu\eta}+u''_{Y_{t\mu}}(S_{t\mu}) u'_{Y_{t\nu}}(S_{t\nu})\delta_{\mu\eta}\\
    &+u'_{Y_{t\mu}}(S_{t\mu}) u''_{Y_{t\nu}}(S_{t\nu}) \delta_{\nu\eta} +u'_{Y_{t\mu}}(S_{t\mu}) u'_{Y_{t\nu}}(S_{t\nu}) u'_{Y_{t\eta}}(S_{t\eta})\,.\nonumber
\end{align}
If $f(\cdot;\bA)$ is a $C^2(\mathbb{R})$ with $P_A$-almost surely bounded first and second derivatives then
\begin{align*}
    &\E[u'_{Y_{t\mu}}(S_{t\mu})\mid S_{t\mu}]=\E[U_{\mu\nu}\mid S_{t\mu},S_{t\nu}]=0\,, \; \E[U_{\mu\nu\eta}\mid S_{t\mu},S_{t\nu},S_{t\eta}]=0\text{ for }\mu\neq \nu\,,\\
    &\E[(u'_{Y_{t\mu}}(S_{t\mu}))^2\mid S_{t\mu}]\,,\,\E[U_{\mu\nu}^2\mid S_{t\mu},S_{t\nu}]\leq C(f)\,,\; \E[U_{\mu\nu\eta}^2\mid S_{t\mu},S_{t\nu},S_{t\eta}]\leq C'(f)\text{ for }\mu\neq\nu
\end{align*}for some positive constants $C(f), C'(f)$. 
\end{lemma}
\begin{proof}
    By definition one has $\E[u'_{Y_{t\mu}}(S_{t\mu})\mid S_{t\mu}]=\int dy  P^x_{  out}(y\mid  S_{t\mu})=0$. For $U_{\mu\nu}$ instead, we first need to realize that $U_{\mu\mu}= P^{xx}_{  out}(Y_{t\mu}\mid S_{t\mu})/ P_{  out}(Y_{t\mu}\mid S_{t\mu})$ which implies $\E[U_{\mu\mu}\mid S_{t\mu}]=\int dy P^{xx}_{  out}(y\mid  S_{t\mu})=0$. Concerning the off-diagonal instead, conditionally on $S_{t\mu},\,S_{t\nu}$ the remaining disorder in the $Y$'s is independent, so for $\mu\neq\nu$ we have $\E[U_{\mu\nu}\mid S_{t\mu},S_{t\nu}]=\int dy P^x_{  out}(y\mid  S_{t\mu})\int dy' P^x_{  out}(y'\mid  S_{t\nu})=0$. Similar arguments hold for $U_{\mu\nu\eta}$ when $\mu\neq \nu$.

    For the boundedness of the derivatives it is sufficient to notice that 
    \begin{align}
        (u'_{Y_{t\mu}}(S_{t\mu}))^2=\Big(\frac{P^{x}_{  out}(Y_{t\mu}|S_{t\mu})}{P_{  out}(Y_{t\mu}|S_{t\mu})}\Big)^2\le C(f)(|Z_\mu|^4+1)\,,
    \end{align}
    the last inequality being true due to  Lemma~\ref{lem:bound_P_out_der} and the fact that $Y_{t\mu}=f(S_{t\mu};\bA_\mu)+\sqrt{\Delta}Z_\mu$. Now it is immediate that after taking the expectation conditioned on $S_{t\mu}$ we obtain a bound $C(f)$. 
    
    In order to deal with $U^2_{\mu\nu}$ we rewrite
    \begin{align*}
    U^2_{\mu\nu}=\Big(\delta_{\mu\nu} \Big(\frac{P^{xx}_{  out}(Y_{t\mu}|S_{t\mu})}{P_{  out}(Y_{t\mu}|S_{t\mu})}-\Big(\frac{P^x_{  out}(Y_{t\mu}|S_{t\mu})}{P_{  out}(Y_{t\mu}|S_{t\mu})}\Big)^2\Big)+\frac{P^x_{  out}(Y_{t\mu}|S_{t\mu})}{P_{  out}(Y_{t\mu}|S_{t\mu})}\frac{P^x_{  out}(Y_{t\nu}|S_{t\nu})}{P_{  out}(Y_{t\nu}|S_{t\nu})}\Big)^2\,.
    \end{align*}
    With the help of Lemma~\ref{lem:bound_P_out_der} one can see immediately that $U^2_{\mu\nu}\le C(f)P(Z_{t\mu},Z_{t\nu})$, where $P$ is some polynomial with even degrees. Once again, after taking expectation we get a bound by a positive constant $C(f)$.
    Similar arguments hold for $\E[U_{\mu\nu\eta}^2\mid S_{t\mu},S_{t\nu},S_{t\eta}]$.
\end{proof}

\section{Details for the proof of Theorem \ref{thm:reduction}}\label{appx:details_th1}
For the sake of clarity the proof has been divided in lemmas. 

\begin{lemma}\label{lem:B_0}
   $B$ in \eqref{eq:B-def} is identically $0$.
\end{lemma}
\begin{proof}
    This is just a consequence of the Nishimori identities in Appendix \ref{app:nishiID}. In fact,
    \begin{align}
        B:=\frac{1}{n}\E_{(t)}\Big\langle \sum_{\mu=1}^{n}u^{\prime}_{Y_{t\mu}}(s_{t\mu})\frac {ds_{t\mu}}{dt} \Big\rangle_t&=\frac{1}{n}\E_{(t)}\sum_{\mu=1}^{n}u^{\prime}_{Y_{t\mu}}(S_{t\mu})\frac {dS_{t\mu}}{dt}\nonumber\\
        &=\frac{1}{n}\sum_{\mu=1}^{n}\E \Big[\EE_{(t)}[u^{\prime}_{Y_{t\mu}}(S_{t\mu})\mid S_{t\mu}]\frac {dS_{t\mu}}{dt}\Big]=0\,,
    \end{align}where we have used Lemma \ref{lem:propertiesPout}.
\end{proof}

Then we turn to the term $A_2$ in \eqref{eq:A_2_def}.

\begin{lemma}
    $A_2=O\Big(\Big(1+\frac{n}{d_m}\Big)\frac{1}{\sqrt{d_m}}\Big)$.
\end{lemma}
\begin{proof}
    In the square parenthesis in \eqref{eq:A_2_def} we integrate by parts $\bv^*$ and $\bW^{*(L)}$, obtaining
    \begin{align}
        A_2=\frac{1}{2}\EE_{(t)}\big(\frac{1}{n}\log\cZ_t-\bar f_{n,t}\big)\sum_{\mu,\nu=1}^n U_{\mu\nu}\rho_L\Big[\rho_L
        -\ba^{*\intercal}
        \frac{\ba^*\circ\varphi'(\balpha_\mu)}{d_L}
        \Big]\frac{\bX_\mu^{(L-1)}\cdot\bX_\nu^{(L-1)}}{d_{L-1}}
    \end{align}where we used the notation $\balpha_\mu=\bW^{*(L)}\bX_\mu^{(L-1)}/\sqrt{d_{L-1}}$. 

    Using Cauchy-Schwartz inequality we can bound $|A_2|$ as in
    \begin{align}
        A_2^2\leq &\frac{\rho_L^2}{4}\mathbb{V}_{(t)}\big[\frac{1}{n}\log\cZ_t\big]\,\EE_{(t)}\sum_{\mu,\nu=1}^n\sum_{\eta,\lambda=1}^n
        U_{\mu\nu}U_{\eta\lambda}
        \frac{\bX_\mu^{(L-1)}\cdot\bX_\nu^{(L-1)}}{d_{L-1}}
        \frac{\bX_\eta^{(L-1)}\cdot\bX_\lambda^{(L-1)}}{d_{L-1}}\nonumber\\
        &\qquad\qquad\qquad \times\Big(\rho_L
        -\ba^{*\intercal}
        \frac{\ba^*\circ\varphi'(\balpha_\mu)}{d_L}
        \Big)\Big(\rho_L
        -\ba^{*\intercal}
        \frac{\ba^*\circ\varphi'(\balpha_\eta)}{d_L}
        \Big)\,.
    \end{align}
    Recall that $\varphi$ is Lipschitz, so $\rho_L$ can be bounded by a constant. In addition, by Theorem~\ref{th:Z_k_concentr}, the variance of the random free entropy contributes with $O(1/n+1/d_m)$ where $d_m=\min\{(d_k)_{k=0}^L\}$. In the summation over the indices $\mu,\nu,\eta,\lambda$ not all terms survive thanks to Lemma \ref{lem:propertiesPout}. The possibilities are $\mu=\nu$ and $\eta=\lambda$, or $\mu<\nu$ and $\eta<\lambda$ with $\mu=\eta$ and $\nu=\lambda$. Another possibility is $\mu=\nu=\eta=\lambda$, which is easily seen to be suppressed for combinatorial reasons. Also, recall that $\EE_{(t)}[U_{\mu\nu}^2\mid S_{t\mu},S_{t\nu}]\leq C(f)$ as proved in Lemma \ref{lem:propertiesPout}.

    Therefore, there exists a constant $C$ such that
    \begin{align}
         A_2^2\leq C\Big(\frac{1}{n}+\frac{1}{d_m}\Big)\EE_{(t)}\sum_{\mu,\nu=1}^n\Big(
         \frac{\bX_\mu^{(L-1)}\cdot\bX_\nu^{(L-1)}}{d_{L-1}}\Big)^2\Big(\rho_L
        -\ba^{*\intercal}
        \frac{\ba^*\circ\varphi'(\balpha_\mu)}{d_L}
        \Big)^2\,.
    \end{align}
    Notice that the only dependency on $\bW^{*(L)}$ is in $\balpha_\mu$. Therefore, considering a brand new set of variables  $Z_i\iid\mathcal{N}(0,1)$ we can recast the above as
    \begin{align}
         A_2^2\!\leq\! C\Big(\frac{1}{n}+\frac{1}{d_m}\Big)\sum_{\mu,\nu=1}^n\EE_{(t)}\Big(
         \frac{\bX_\mu^{(L-1)}\cdot\bX_\nu^{(L-1)}}{d_{L-1}}\Big)^2\EE_{\bZ,\ba^*}\Big(\rho_L
        -\sum_{i=1}^{d_L}\frac{(a_i^*)^2}{d_L}\varphi'\Big( Z_i\sqrt{\frac{\|\bX_\mu^{(L-1)}\|^2}{d_{L-1}}}\Big)
        \Big)^2\!.
    \end{align}
    With an exponentially small cost $O(e^{-cd_{L-2}\delta})$, for a given $\delta$ and some constant $c>0$ thanks to Lemma~\ref{lem:concentration_propagation}, one can assume that $\|\bX_\mu^{(\ell)}\|^2/d_{\ell}\geq 
    \EE\|\bX_\mu^{(\ell)}\|^2/d_{\ell}-\delta$, for $\ell=0,\dots,L-1$. Then, using Lemma~\ref{lem:orthogonality propagation}
    \begin{align}
        \EE\Big(\frac{\bX_\mu^{(L-1)}\cdot\bX_\nu^{(L-1)}}{d_{L-1}}\Big)^{4}\leq \frac{K}{d^2_{[0:d_{L-1}]}}
    \end{align}
    for an appropriate constant $K>0$. 
    Furthermore, it is easy to check that replacing the $(a_i^*)^2$'s by $1$ costs an overall $O(d_L^{-1})$. Also,
    \begin{align}
        \EE\Big(\rho_L
        -\frac{1}{d_L}\sum_{i=1}^{d_L}\varphi'\Big( Z_i\sqrt{\frac{\|\bX_\mu^{(L-1)}\|^2}{d_{L-1}}}\Big)
        \Big)^4\leq \frac{C}{d^2_{[0:d_{L-1}]}}
    \end{align}by Lemma~\ref{lem:moments_post_activations}. Then, using Cauchy-Schwartz inequality one readily gets:
    \begin{align}
        A_2^2=O\Big(\Big(1+\frac{n}{d_m}\Big)^2\frac{1}{d_m}\Big)\,.
    \end{align}    
\end{proof}

\begin{lemma}
    Recall the definition of $A_1^{  off}$ in \eqref{eq:A_1_off_def}. $A_1^{  off}=O\big((1+n/d_m)\sqrt{n}/d_m\big)$.
\end{lemma}
\begin{proof}
Let us take \eqref{eq:A_1off} as a starting point. Define the two contributions:
\begin{align}
    A_{11}^{  off}(\tau)&=\,\EE_{(t)}\Big(\frac{1}{n}\log\cZ_t-\bar f_{n,t}\Big)\sum_{\mu<\nu,1}^n\sum_{\eta=1}^nU_{\mu\nu\eta}\frac{\bX_\mu^{(L-1)}\cdot\bX_\eta^{(L-1)} }{d_{L-1}}\times \nonumber\\
    &\qquad\qquad \qquad\qquad\times\sum_{i=1}^{d_L}
    \frac{a_i^*\varphi'(\alpha_{i\eta})}{\sqrt{d_L}} 
    \frac{\varphi'(\alpha_{i\mu}(\tau))-\rho_L}{d_L} \varphi(\alpha_{i\nu})
\end{align}
and
\begin{align}
    A_{12}^{  off}(\tau)&=\,\EE_{(t)}\Big(\frac{1}{n}\log\cZ_t-\bar f_{n,t}\Big)\sum_{\mu<\nu,1}^n U_{\mu\nu}\frac{\bX_\mu^{(L-1)}\cdot\bX_\nu^{(L-1)} }{d_{L-1}}\sum_{i=1}^{d_L}\frac{\varphi'(\alpha_{i\mu}(\tau))-\rho_L}{d_L}\varphi'(\alpha_{i\nu})\,.
\end{align}Then $A_1^{  off}=\int_0^{\pi/2}\cos\tau\big(A_{11}^{  off}(\tau)+A_{12}^{  off}(\tau)\big)d\tau$.

Let us start from $A_{11}^{  off}(\tau)$. Using Cauchy-Schwartz's inequality, and $\EE[U_{\mu\nu\eta}\mid S_{t\mu},S_{t\nu},S_{t\eta}]=0$ together with $\EE[U_{\mu\nu\eta}^2\mid S_{t\mu},S_{t\nu},S_{t\eta}]\leq C'(f)$, one can bound $(A_{11}^{  off}(\tau))^2$ as
\begin{align}
    (A_{11}^{  off}(\tau))^2&\leq C''(f)\mathbb{V}_{(t)}\Big[\frac{1}{n}\log\cZ_t\Big]\sum_{\mu<\nu,1}^n
    \sum_{\eta=1}^n\EE_{(t)} \Big(\frac{\bX_\mu^{(L-1)}\cdot\bX_\nu^{(L-1)} }{d_{L-1}} \Big)^2
    \nonumber\\
    &\qquad\qquad\qquad\times\frac{1}{d_L^2}\sum_{i=1}^{d_L}\frac{(\varphi'(\alpha_{i\mu}(\tau))-\rho_L)^2}{d_L}(X^{(L)}_{i\nu})^2
    (\varphi'(\alpha_{i\eta}))^2
\end{align}
where we have computed the expectation over the $\ba^*$'s. Since $\varphi$ is Lipschitz, we can easily continue the bound as in:
\begin{align}
    (A_{11}^{  off}(\tau))^2&\leq C'''(f)\mathbb{V}_{(t)}\Big[\frac{1}{n}\log\cZ_t\Big]\sum_{\mu<\nu,1}^n
    \sum_{\eta=1}^n\EE_{(t)} \Big(\frac{\bX_\mu^{(L-1)}\cdot\bX_\nu^{(L-1)} }{d_{L-1}} \Big)^2\frac{\|\bX_\nu^{(L)}\|^2}{d_L^3}\,.
\end{align}
By splitting the last expectation again by Cauchy-Schwartz and Lemmas~\ref{lem:moments_post_activations} and \ref{lem:orthogonality propagation} we have
\begin{align}
    (A_{11}^{  off}(\tau))^2=O\Big(\Big(1+\frac{n}{d_m}\Big)^2\frac{n}{d_m^2}\Big)
\end{align}where we have taken into account possible matchings $\eta=\mu$ or $\eta=\nu$. The above bound is uniform in $\tau$.

Analogous arguments apply to $A_{12}^{  off}(\tau))$:
\begin{align}
    (A_{12}^{  off}(\tau))^2&=C(f) \mathbb{V}_{(t)}\Big[\frac{1}{n}\log\cZ_t\Big]\sum_{\mu<\nu,1}^n \EE_{(t)} \Big(\frac{\bX_\mu^{(L-1)}\cdot\bX_\nu^{(L-1)} }{d_{L-1}}\Big)^2(\mathbbm{1}(B_L)+\mathbbm{1}(B_L^c))\nonumber\\
    &\qquad\qquad\qquad\times \sum_{i,j=1}^{d_L}\frac{(\varphi'(\alpha_{i\mu}(\tau))-\rho_L)\varphi'(\alpha_{i\nu})}{d_L}
    \frac{(\varphi'(\alpha_{j\mu}(\tau))-\rho_L)\varphi'(\alpha_{j\nu})}{d_L}
    \,.
\end{align}In the above we have introduced the event 
\begin{align}
    B_L:=\{\|\bX_\mu^{(L-1)}\|^2\geq d_{L-1}\sigma_{L-1}/2\}
\end{align}that, thanks to \eqref{eq:Markov_moments_post_activ} with $k=2$ and $\delta=\sigma_{L-1}/2$, occurs with probability at least $1-O((d_{[0:L-1]})^{-2})$. Since the double sum is bounded by a constant thanks to the Lipschitzness of $\varphi$, it is easy to see that the overall contribution deriving from $\mathbbm{1}(B_L^c)$ is $O((1/n+1/d_m)n^2/d_m^2)$.

We thus shift our focus on the contribution deriving from $\mathbbm{1}(B_L)$. The diagonal terms, $i=j$, in the double sum above are easily bounded by $O(d_L^{-1})$ because $\varphi$ is Lipschitz. Let us analyse instead the off-diagonal terms for $i\neq j$. The expectation w.r.t. $\EE_{\bW^{*(L)},\tilde\bW^*}$ only can be pushed till the mentioned double sum, and for $i\neq j$ it yields an overall contribution:
\begin{align}
    \frac{d_L(d_L-1)}{d_L^2}\big(\EE_{\bW^{*(L)},\tilde\bW^*}[(\varphi'(\alpha_{1\mu}(\tau))-\rho_L)\varphi'(\alpha_{1\nu})]\big)^2 \,.
\end{align}
The above is an expectation of a bounded random variable w.r.t.\ two Gaussians with correlation
\begin{align}
    \EE_{\bW^{*(L)},\tilde\bW^*} \alpha_{1\mu}(\tau)\,
    \alpha_{1\nu}= \sin\tau\frac{\bX_\mu^{(L-1)}\cdot\bX_\nu^{(L-1)}}{d_{L-1}}\,.
\end{align}
Therefore, by Lemma~\ref{lem:approx}
\begin{align}
    \EE_{\bW^{*(L)},\tilde\bW^*}[(\varphi'(\alpha_{1\mu}(\tau))-\rho_L)\varphi'(\alpha_{1\nu})]&=\EE_{\bW^{*(L)},\tilde\bW^*}[(\varphi'(\alpha_{1\mu}(\tau))-\rho_L)]\EE_{\bW^{*(L)}}\varphi'(\alpha_{1\nu}) \nonumber\\
    &+O\Big(\frac{\bX_\mu^{(L-1)}\cdot\bX_\nu^{(L-1)}}{\|\bX_\nu^{(L-1)}\|^2}\Big).
\end{align}
Since $\varphi$ is Lipschitz, and following from the proof of Lemma~\ref{lem:moments_post_activations}, restricted to the event $B_L$ we have
\begin{align}
    \EE_{\bW^{*(L)},\tilde\bW^*}[(\varphi'(\alpha_{1\mu}(\tau))-\rho_L)\varphi'(\alpha_{1\nu})]=O\Big(\Big({\frac{\|\bX_\mu^{(L-1)}\|^2}{d_{L-1}}}-{\sigma_{L-1}}\Big)+\frac{\bX_\mu^{(L-1)}\cdot\bX_\nu^{(L-1)}}{d_{L-1}}\Big).
\end{align}
By plugging this back into the inequality for $(A_{12}^{  off}(\tau))^2$, and using Lemmas \ref{lem:orthogonality propagation} and \ref{lem:moments_post_activations}, the event $B_L$ contributes for $O((1+n/{d_m})n/{d_m^2})$.

The statement follows from selecting the worst of the convergence rates, that is the one of $A_{11}^{  off}(\tau)$.

\end{proof}

\begin{lemma}\label{lem:A_3-A_1}
        $A_3-A_1=\Big(\Big(1+\sqrt{\frac{n}{d_m}}\Big)\frac{1}{\sqrt{d_m}}\Big)$.
\end{lemma}
\begin{proof}
    We take up the computation from \eqref{eq:main_A3-A1}. We have already proved that its second line is controlled with $O((1+n/d_m)\sqrt{n}/d_m)$. To bound the first line of \eqref{eq:main_A3-A1}, that we denote by $D$ hereby, it is sufficient to use Cauchy-Schwartz's inequality first, and Lemma~\ref{lem:propertiesPout}:
    \begin{align}
        D^2\leq Cn\mathbb{V}_{(t)}\Big[\frac{1}{n}\log\mathcal{Z}_t\Big]\EE\Big[\epsilon_L-\frac{\|\bX_1^{(L)}\|^2}{d_L} +\rho_L\frac{\balpha_1^\intercal\varphi(\balpha_1)}{d_L} \Big]^2
    \end{align}for a positive constant $C>0$. Thanks to Lemma~\ref{lem:moments_post_activations} we can replace ${\|\bX_1^{(L)}\|^2}/{d_L}$ and ${\balpha_1^\intercal\varphi(\balpha_1)}/{d_L}$ with their expectations up to an $O(d_m^{-1})$:
    \begin{align}
        D^2\leq Cn\mathbb{V}_{(t)}\Big[\frac{1}{n}\log\mathcal{Z}_t\Big]\Big[O(d_m^{-1})+(\epsilon_L-\sigma_L+\rho_L^2\sigma_{L-1} )^2\Big]=O\Big(\frac{1}{d_m}\Big(1+\frac{n}{d_m}\Big)\Big),
    \end{align}where we used the definition of $\epsilon_L$. Combining the rate for $D$ and the one for $A_{1}^{  off}$ we thus get the statement.
\end{proof}

So far, we have proved that the free entropies of the original model and the reduced are as closed as desired. It is finally time to deal with the first term of \eqref{eq:MI_definition}.

\begin{lemma}\label{lem:Psi_constant_MI}
Define
    \begin{align}
        \Psi_k:=\EE\log P_{  out}\Big(Y^{(k)}_1\mid \eta_k\frac{\ba^{* (k)\intercal}\bX^{(k)}_1}{\sqrt{d_k}}+\sqrt{\gamma_k}\xi^{* (k)}_{1}\Big)
    \end{align}where $\ba^{* (k)}\in\mathbb{R}^{d_k}$,  $k=L$ or $L-1$, and $\eta_k,\gamma_k$ are defined in \eqref{eq:coeffs}. Recall that $$Y_1^{(k)}\sim P_{out}\Big(\cdot\mid \eta_k\frac{\ba^{*(k)\intercal}\bX^{(k)}_1}{\sqrt{d_k}}+\sqrt{\gamma_k}\xi^{* (k)}_{1}\Big).$$ 
    Then
    \begin{align}
        |\Psi_L-\Psi_{L-1}|=O(d_m^{-1/2}).
    \end{align}
\end{lemma}
\begin{proof}
    Notice that $\ba^{* (k)}$ is independent on $\xi^{* (k)}$. Therefore, letting $Z\sim\mathcal{N}(0,1)$, we have
    \begin{align}
        \Psi_k=\EE \int dy P_{out}\Big(y\mid Z\sqrt{\eta_k^2\frac{\|\bX^{(k)}_1\|^2}{d_k}+\gamma_k}\Big)\log P_{out}\Big(y\mid Z\sqrt{\eta_k^2\frac{\|\bX^{(k)}_1\|^2}{d_k}+\gamma_k}\Big)\,.
    \end{align}
    Recall $\rho=1m,\epsilon=0$ out of convenience, without loss of generality. Then $\eta_L=1,\gamma_L=0$.
    It is then convenient to introduce
    \begin{align}
        S(t)=t\sqrt{\frac{\|\bX^{(L)}_1\|^2}{d_L}}+(1-t)\sqrt{\rho_L^2\frac{\|\bX^{(L-1)}_1\|^2}{d_{L-1}}+\epsilon_L}
    \end{align}and
    \begin{align}
        \Psi(t)=\EE \int dy P_{out}(y\mid ZS(t))\log P_{out}(y\mid ZS(t))\,.
    \end{align}It is then straightforward to see that $\Psi(0)=\Psi_{L-1}$ and $\Psi(1)=\Psi_L$.
    Using Lemma~\ref{lem:propertiesPout} the one can readily show that
    \begin{align}
        |\dot{\Psi}(t)|\leq C(f)\EE|Z|\EE|\dot{S}(t)|
    \end{align} for some $C(f)>0$. Using Lemma~\ref{lem:moments_post_activations}, and the fact that $\rho_L\sigma_{L-1}+\epsilon_L=\sigma_L$ one readily gets the result.
    
\end{proof}

For completeness we also prove the following Lemma, which relates the various mutual informations
\begin{lemma}
    \label{lem:relation_MIs}
    Recall the definition
    \begin{align}
        I_n(\bxi^*;\mathcal{D}\mid\btheta^*)= H(\mathcal{D}\mid\btheta^*)&-H(\mathcal{D}\mid\bxi^*,\btheta^*)=-\EE\log \EE_\xi P_{out}\Big(Y_1\mid \rho\frac{\ba^{*\intercal}\bX^{(L)}_1}{\sqrt{d_L}}+\sqrt{\epsilon}\xi\Big) \nonumber\\
        &+\EE\EE_{\xi^*}\log  P_{out}\Big(Y_1\mid \rho\frac{\ba^{*\intercal}\bX^{(L)}_1}{\sqrt{d_L}}+\sqrt{\epsilon}\xi^*\Big) 
    \end{align}where $\xi,\xi^*\sim\mathcal{N}(0,1)$ and $Y_1$ is distributed as in \eqref{eq:output_channel_labels}. Then
    \begin{align}
        I_n(\bxi^*;\mathcal{D}\mid\btheta^*)&=-\int dy \EE_{Z,\xi^*}P_{out}(y\mid Z\sqrt{\rho^2\sigma_L}+\sqrt{\epsilon}\xi^*)\log \EE_\xi
        P_{out}(y\mid Z\sqrt{\rho^2\sigma_L}+\sqrt{\epsilon}\xi) \nonumber\\
        &+\int dy \EE_{Z}P_{out}(y\mid Z\sqrt{\rho^2\sigma_L+\epsilon})\log
        P_{out}(y\mid Z\sqrt{\rho^2\sigma_L+\epsilon})+O(d_m^{-1/2})\label{eq:MIconstant_terms_appendix},
    \end{align}
    where $Z\sim\mathcal{N}(0,1)$.
\end{lemma}
\begin{proof}
    For both terms the proof proceeds in the same way, by observing that
    \begin{align}
        \rho\frac{\ba^{*\intercal}\bX^{(L)}_1}{\sqrt{d_L}}\sim Z\sqrt{\rho^2\frac{\|\bX^{(L)}_1\|^2}{d_L}}
    \end{align}where $\sim$ denotes equality in distribution. Then one proceeds interpolating as in Lemma~\ref{lem:Psi_constant_MI}, and using the concentration in Lemma~\ref{lem:moments_post_activations}. Note that for $H(\mathcal{D}\mid\bxi^*,\btheta^*)$ it possible to join the Gaussians $Z$ and $\xi^*$ in a unique Gaussian, whereas for $H(\mathcal{D}\mid\btheta^*)$ this is not possible which explains the presence of the two variables $\xi,\xi^*$ in the first line of \eqref{eq:MIconstant_terms_appendix}.
\end{proof}

\section{Complete network reduction} \label{app:complete_reduction}
Iterating Theorem \ref{thm:reduction} it is possible to reduce the full network to a GLM in terms of mutual information between dataset and teacher weights. To prove Corollary~\ref{thm:MIAllequivalence} holds with the coefficients \eqref{eq:coeffs} it is sufficient to start from a network with $k$ layers left, and infer the coefficients $\eta_{k-1},\gamma_{k-1}$ starting from $\eta_k,\gamma_k$.

By the Layer Reduction Theorem we can replace
\begin{align}
    \eta_k \frac{\mathbf{a}^{*\intercal}\bX_\mu^{(k)}}{\sqrt{d_k}}+\sqrt{\gamma_k}\zeta_\mu^{* (k)}
    \longrightarrow 
    \eta_k\Big( \rho_k\frac{\mathbf{v}^{*\intercal}\bX_\mu^{(k-1)}}{\sqrt{d_{k-1}}}+\sqrt{\epsilon_k}\xi_\mu^*\Big)+\sqrt{\gamma_k}\zeta_\mu^{*(k)},
\end{align}
where $\ba^*\in\mathbb{R}^{d_k},\bv^*\in\mathbb{R}^{d_{k-1}}$ are two vectors with i.i.d.\ standard Gaussian components, and $\zeta_\mu^{*(k)},\xi_\mu^*$ are two standard Gaussians. Hence we infer the recursion
\begin{align}
    \eta_{k-1}=\eta_k\rho_k\,,\quad 
    \gamma_{k-1}=\eta_k^2\epsilon_k+\gamma_k\,.
\end{align}Using $\eta_L=\rho,\gamma_L=\epsilon$ then one readily gets \eqref{eq:coeffs} from the above recursion.

\section{Concentration of free energy}\label{app:concentr}
In this section we prove the concentration of the free energy for the interpolated model described in Section~\ref{subsec:proof_red}. For convenience we replicate here some of the main notations. We consider interpolating models of teacher and student between the nonlinear argument of the readout function for the $L$-layer neural network and its corresponding linearized model:
    \begin{align}
    S_{t\mu} &:= \sqrt{1-t}\frac{\ba^{*\intercal}}{\sqrt{d_L}}  \varphi\Big(\frac{\bW^{* (L)} \ \bX_\mu^{(L-1)}}{\sqrt{d_{L-1}}}\Big) + \sqrt{t}\rho_L \frac{\bv^{*  \intercal}\bX_\mu^{(L - 1)}}{\sqrt{d_{L - 1}}}+\sqrt{t \epsilon_L}\zeta_\mu^{* (L)}\,,\\
    s_{t\mu}&:= \sqrt{1-t}\frac{\ba^{ \intercal}}{\sqrt{d_{L}}}  \varphi\Big(\frac{\bW^{(L)} \ \bx_\mu^{(L-1)}}{\sqrt{d_{L-1}}}\Big) + \sqrt{t}\rho_L \frac{\bv^{ \intercal}\bx_\mu^{(L - 1)}}{\sqrt{d_{L - 1}}}+\sqrt{t \epsilon_L}\zeta_\mu^{ (L)}\,,
    \end{align} 
    and
    \begin{align}
        Y_{t\mu}=f(S_{t\mu};\bA_{\mu}) +\sqrt{\Delta}Z_{\mu}\,.
    \end{align}
   We recall the corresponding random free entropy $f_{n,t}$ and partition function $\cZ_t$
\begin{align}
    f_{n,t}=\frac{1}{n}\log\cZ_t=\frac{1}{n}\log\int dP(\btheta^{(L)})\EE_{\bv}\prod_{\mu=1}^n\EE_{\zeta_\mu^{(L)}} P_{out}(Y_{t\mu}|s_{t\mu})\,.
\end{align}
   The main result of this section is
   \begin{theorem}\label{th:Z_k_concentr_apx}
For the activation and readout function $\varphi,\,f$ taken as in Theorem~\ref{thm:reduction}, there exists a non-negative constant $C(f,\varphi)$ such that
\begin{align}
  \mathbb{V}\Big(\dfrac{1}{n}\log \cZ_t\Big)
    \leq C(f,\varphi)\Big(\dfrac{1}{n}+\frac{1}{d_m}\Big)\,.
\end{align}
\end{theorem}
We remark that Theorem~\ref{th:Z_k_concentr} is a direct consequence of the latter result, where $t$ is taken equal to $0$.
   
We  prove  Theorem~\ref{th:Z_k_concentr_apx} in several steps, first we show concentration   with respect to the set of parameters $\psi_1:=\{Z_\mu,\zeta^{*(L)}_\mu,\mathbf{X}_\mu^{(0)}\}_{\mu=1}^n$ using  classical Poincar\'e-Nash inequality, after this, with respect to $\psi_2:=\{\bA_\mu\}_{\mu}^n$ using the corollary of Efron-Stein inequality, and finally, with  respect to  $\psi_3:=\{(\mathbf{W}^{*(k)})_{\ell=1}^L,\bv^{*},\ba^*\}$. 
We recall two classical concentration results, whose proofs can be found in \cite{Boucheron2004}, Chapter 3.
\begin{proposition}[Poincar\'e-Nash inequality]
  \label{prop:poincare}
	Let $\xi=[\xi_1,\ldots,\xi_K]^\intercal$ be a real Gaussian standard random vector. If $g:\R^K\mapsto \R$ is a continuously  differentiable function, then
\begin{align}\label{ineq:p-n}
	\mathbb{V} g(\xi)\le\E\|\nabla g(\xi)\|^2\,.
\end{align}
\end{proposition}

\begin{proposition}[Bounded difference]
 \label{prop:efron-stein}   
 Let $\xi=[\xi_1,\ldots,\xi_K]^\intercal$ be a random vector with i.i.d. elements taking values in some space $\mathcal{A}$. If function $g:\mathcal{A}^K\mapsto \R$ satisfies 
 \begin{align}
     \sup_{1\leq i\leq K}\sup_{x_1,\ldots,x_K, x\prime_i\in \mathcal{A}}|g(x_1,\ldots,x_i,\ldots, x_K)-g(x_1,\ldots,x^\prime_i,\ldots, x_K)|
     \leq C
 \end{align} for some $C>0$, then 
\begin{align}\label{ineq:e-s}
	\mathbb{V}\{g(\xi)\}\le\dfrac{1}{4}KC^2.
\end{align}
\end{proposition}
In what follows we  denote $P^y(y|x):=\partial_y P_{  out}(y|x)$ and $P^x(y|x):=\partial_x  P_{  out}(y|x)$. 
\begin{lemma}
\label{lemma:concentration_free}
  There exists a non-negative constant $C(f,\varphi$) such that
  \begin{align*}
      \mathbb{E}\Big(\frac{1}{n}\log \cZ_t-\mathbb{E}_{\psi_1}\frac{1}{n}\log \cZ_t\Big)^2\leq \frac{C(f,\varphi)}{n}\,.
  \end{align*}
\end{lemma}
\begin{proof}
    Since  the elements of $\bX_{\mu}^{(0)}$, $\zeta^{*(L)}_{\mu}$, and $Z_{\mu}$ for $\mu=1,\ldots,n$ are jointly independent standard Gaussian, we can use Proposition~\ref{prop:poincare} and write
    \begin{align*}
        &\mathbb{E}\mathbb{V}_{\psi_1}\Big(\frac{1}{n}\log \cZ_t\Big)\leq \frac{1}{n^2} \EE\|\nabla_{\psi_1}\log\cZ_t\|^2\\
        &=\!\frac{1}{n^2}\sum_{\mu=1}^n\EE\Big(\frac{\partial\log\cZ_t}{\partial\zeta^{*(L)}_\mu}\Big)^2+\frac{1}{n^2}\sum_{\mu=1}^n\sum_{i=1}^{d_{0}}\EE\Big(\frac{\partial\log\cZ_t}{\partial X^{(0)}_{i\mu}}\Big)^2+\frac{1}{n^2}\sum_{\mu=1}^n\EE\Big(\frac{\partial\log\cZ_t}{\partial Z_\mu}\Big)^2=:I_1+I_2+I_3.
    \end{align*}
For brevity in what follows we write $P_\mu=P_{out}(Y_{t\mu}|s_{t\mu})$. Then,  expressing the derivatives from the first term, we get
\begin{align*}
    \Big|\frac{\partial\log\cZ_t}{\partial\zeta^{*(L)}_\mu}\Big|\leq \Big|\Big\langle \frac{P_{\mu}^y}{P_\mu}\Big\rangle f^\prime(S_{t\mu};\bA_\mu)\sqrt{t\epsilon_L}\Big|\leq \sqrt{t\epsilon_L}C(f)(|Z_\mu|^2+1)\,.
\end{align*}
    The last inequality is true due to Lemma~\ref{lem:bound_P_out_der} and the fact that function $f^\prime$ is bounded. Since $Z_\mu$ is Gaussian, all the moments of its absolute value are bounded with a constant, which gives us
    \begin{align*}
        \EE\Big(\frac{\partial\log\cZ_t}{\partial\zeta^{*(L)}_\mu}\Big)^2\leq C(f)\Rightarrow I_1\leq \frac{C(f)}{n}\,.
    \end{align*}
    The partial derivative with respect to $Z_\mu$ gives a simple expression that also can be handled using Lemma~\ref{lem:bound_P_out_der}
    \begin{align}
        \Big|\frac{\partial\log\cZ_t}{\partial Z_\mu}\Big|=\sqrt{\Delta}\Big|\Big\langle\frac{P_\mu^y}{P_\mu}\Big\rangle\Big|\leq C(f)(|Z_\mu|^2+1).
    \end{align}
    This gives us immediately $I_3\leq C(f)n^{-1}$.
    
   Before passing to the more challenging part, let us set some notation.
   First, set $\alpha^{*(\ell-1)}_{i\mu}:=\bW^{*(\ell)}_i\cdot\bX^{(\ell-1)}_\mu/\sqrt{d_{\ell-1}}$.
   In what follows we will need to calculate derivatives of $\bX_\nu^{(\ell)}$ and $\bx_\nu^{(\ell)}$. The chain rule gives us
    \begin{align*}
        \frac{\partial X_{j\nu}^{(\ell)}}{\partial X_{i\mu}^{(0)}}
        &=\frac{\partial\varphi\Big(\frac{\bW_j^{*(\ell)}}{\sqrt{d_{\ell-1}}}\varphi\Big(\frac{\bW^{*(\ell-1)}}{\sqrt{d_{\ell-2}}}\varphi\Big(\ldots\varphi\Big(\frac{\bW^{*(1)}\bX_{\mu}^{(0)}}{\sqrt{d_0}}\Big)\Big)}{\partial X_{i\mu}^{(0)}}\\
        &=\delta_{\mu,\nu}\Big(\prod_{p=1}^{\ell-1}\sum_{j_p=1}^{d_p}\Big)\varphi^\prime(\alpha^{*(\ell-1)}_{j\nu})\frac{W^{*(\ell)}_{jj_{\ell-1}}}{\sqrt{d_{\ell-1}}}\varphi^\prime(\alpha^{*(\ell-2)}_{j_{\ell-1}\nu})\frac{W^{*(\ell-1)}_{j_{\ell-1}j_{\ell-2}}}{\sqrt{d_{\ell-2}}}\ldots\varphi^\prime(\alpha^{*(0)}_{j_1\nu})\frac{W^{*(1)}_{j_1i}}{\sqrt{d_{0}}}\,.
    \end{align*}
    For $1\leq \ell_1\leq \ell_2\leq L$ we denote $d_{\ell_2}\times d_{\ell_1-1}$ matrix
    \begin{align}
        \bB^{*[\ell_1:\ell_2]}_{\nu}=\frac{\boldsymbol{\Phi}^{*(\ell_2-1)}_\nu\bW^{*(\ell_2)}\boldsymbol{\Phi}^{*(\ell_2-2)}_\nu\ldots\boldsymbol{\Phi}^{*(\ell_1-1)}_\nu\bW^{*(\ell_1)}}{\sqrt{d_{\ell_2-1}\ldots d_{\ell_1-1}}}
    \end{align}
    and $d_{\ell+1}\times d_{\ell+1}$ diagonal matrix $\boldsymbol{\Phi}^{*(\ell)}_\nu=\text{diag}(\varphi^\prime(\alpha^{*(\ell)}_{1\nu}),\ldots,\varphi^\prime(\alpha^{*(\ell)}_{d_{\ell-1}\nu}))$. Then previous expression can be rewritten as
    \begin{align}\label{eq:part_der_X_X}
        \frac{\partial X_{j\nu}^{(\ell)}}{\partial X_{i\mu}^{(0)}} =\delta_{\mu,\nu}B^{*[1:\ell]}_{ji,\nu}\,.
    \end{align}
       For the following calculations, it  will be useful to find the order of the operator norm of $\bB^{*[\ell_1:\ell_2]}_{\nu}$. Since the first derivative of $\varphi$ is bounded by some constant $C(\varphi)$, it is straightforward that the operator norm $\|\boldsymbol{\Phi}^{*(\ell)}_\nu\|$ is also bounded with $C(\varphi)$. Lastly, the expected operator norm of high dimensional matrices, say of size $a\times b$, with i.i.d.\ standard Gaussian elements is of order $O(\sqrt{a})=O(\sqrt{b})$ with high probability. This entails
    \begin{align}\label{eq:bA_bound}
        \EE\|\bB^{*[\ell_1:\ell_2]}_{\nu}\|\leq \bar C(\varphi)\sqrt{\prod_{p=\ell_1}^{\ell_2}\frac{\EE\|\bW^{*(p)}\|^2}{d_{p-1}}}=O(1)\,.
    \end{align} Similar bounds hold for any finite moment of the operator norm.
    Following the same lines we can write the expression for the derivative of $ x_{j\nu}^{(\ell)}$ and define  the students version of $\boldsymbol{\Phi}^{*(\ell)}_\nu$ and $\bB^{*[\ell_1:\ell_2]}_{\nu}$ by removing stars from all involved parameters. 
        \begin{align}\label{eq:part_der_x_X}
        \frac{\partial x_{j\nu}^{(\ell)}}{\partial X_{i\mu}^{(0)}} =\delta_{\mu,\nu}B^{[1:\ell]}_{\nu,ji}\,.
    \end{align}
    Returning to $I_2$, we have
    \begin{align*}
        &\frac{\partial\log \cZ_t}{\partial X_{i\mu}^{(0)}}=\sum_{j=1}^{d_{L-1}}\frac{\partial X_{j\mu}^{(L-1)}}{\partial X_{i\mu}^{(0)}}\Big\langle\frac{P_{\mu}^y}{P_{\mu}}\Big\rangle f^\prime(S_{t\mu};\bA_\mu)\Big(\sqrt{1-t}\frac{(\ba^{*\intercal}\boldsymbol{\Phi}^{*(L-1)}_\mu\bW^{*(L)})_j}{\sqrt{d_Ld_{L-1}}}+\sqrt{t}\rho_{L}\frac{v_j^{*}}{\sqrt{d_{L-1}}}\Big)\\
        &\qquad\qquad+\sum_{j=1}^{d_{L-1}}\Big( \sqrt{1-t}\Big\langle\frac{\partial x_{j\mu}^{(L-1)}}{\partial X_{i\mu}^{(0)}}\frac{P_{\mu}^x}{P_{\mu}}\frac{(\ba^{\intercal}\boldsymbol{\Phi}^{(L-1)}_\mu\bW^{(L)})_j}{\sqrt{d_Ld_{L-1}}}\Big\rangle
        +\sqrt{t}\rho_L\Big\langle\frac{\partial x_{j\mu}^{(L-1)}}{\partial X_{i\mu}^{(0)}}\frac{P_{\mu}^x}{P_{\mu}}\frac{v_j}{\sqrt{d_{L-1}}}\Big\rangle\Big).
        \end{align*}
Plugging  (\ref{eq:part_der_X_X}) and (\ref{eq:part_der_x_X}) in the latter expression  and taking the sum over $j$ we obtain
    \begin{align*}
        \frac{\partial\log \cZ_t}{\partial X_{i\mu}^{(0)}}
        =\Big\langle\frac{P_{\mu}^y}{P_{\mu}}\Big\rangle f^\prime(S_{t\mu};\bA_\mu)\Big(\sqrt{1-t}\frac{(\ba^{*\intercal}\bB_{\mu}^{*[1:L]})_i}{\sqrt{d_L}}+\sqrt{t}\rho_L\frac{(\bv^{*\intercal}\bB_{\mu}^{*[1:L-1]})_i}{\sqrt{d_{L-1}}}\Big)\\
        + \sqrt{1-t}\Big\langle\frac{P_{\mu}^x}{P_{\mu}}\frac{(\ba^\intercal\bB_{\mu}^{[1:L]})_i}{\sqrt{d_L}}\Big\rangle
        +\sqrt{t}\rho_L\Big\langle\frac{P_{\mu}^x}{P_{\mu}}\frac{(\bv^\intercal\bB_{\mu}^{[1:L-1]})_i}{\sqrt{d_{L-1}}}\Big\rangle\,.
        \end{align*}
        Using twice the simple inequality $(a+b)^2\leq 2(a^2+b^2)$ as well as Lemma~\ref{lem:bound_P_out_der} and boundedness of $f^\prime$ we get thee expression for $I_2$
      \begin{align*}
        I_2
        \leq \frac{1}{n^2}\sum_{\mu=1}^n\Big[C(f)\EE\Big((1+|Z_\mu|^2)^2\big(\|\frac{\ba^{*\intercal}}{d_k}\bB_{\mu}^{*[1:L]}\|^2+\|\frac{\bv^{*\intercal}}{d_{L-1}}\bB_{\mu}^{*[1:L-1]}\|^2\big)\Big)\\
        + (1-t)\sum_{i=1}^{d_0}\EE\Big\langle\frac{P_{\mu}^x}{P_{\mu}}\frac{(\ba^\intercal\bB_{\mu}^{[1:L]})_i}{\sqrt{d_L}}\Big\rangle^2
        +t\rho_L^2\sum_{i=1}^{d_0}\EE\Big\langle\frac{P_{\mu}^x}{P_{\mu}}\frac{(\bv^\intercal\bB_{\mu}^{[1:L-1]})_i}{\sqrt{d_{L-1}}}\Big\rangle^2\Big]\,.
        \end{align*}   
For the last two terms, we can bring square inside the Gibbs brackets with the help of Jensen inequality and then we notice that obtained expressions inside the Gibbs brackets depend only on $Y_{t\mu}$ and $s_{t\mu}$, which allows us to use Nishimori identity by removing brackets and adding $*$ to all the parameters. After this we can bound each ratio which includes $P_{out}$ by using again Lemma~\ref{lem:bound_P_out_der}. All this will result in
  \begin{align*}
     I_2
        \leq \frac{C(f)}{n^2}\sum_{\mu=1}^n\Big[\frac{\EE(1+|Z_\mu|^2)^2\|\ba^{*\intercal}\bB_{\mu}^{*[1:L]}\|^2}{d_L}+\frac{\EE(1+|Z_\mu|^2)^2\|\bv^{*\intercal}\bB_{\mu}^{*[1:L-1]}\|^2}{d_{L-1}}\Big]\,.
        \end{align*}   
Since $Z_\mu$, $\ba^*$, and $\bv^{*}$ are Gaussian and independent of $\bB_{\mu}^{*[1:L]}$ we have
  \begin{align*}
     I_2
        &\leq \frac{C(f)}{n^2}\sum_{\mu=1}^n\Big[\frac{\EE(1+|Z_\mu|^2)^2\EE\|\ba^{*}\|^2\EE\|\bB_{\mu}^{*[1:L]}\|^2}{d_L}\\
        &\qquad\qquad\qquad +\frac{\EE(1+|Z_\mu|^2)^2\EE\|\bv^{*}\|^2\EE\|\bB_{\mu}^{*[1:L-1]}\|^2}{d_{L-1}}\Big]\leq\frac{ C(f,\varphi) }{n}\,.
        \end{align*}   
The last inequality being true due to (\ref{eq:bA_bound}) and the fact that $\EE\|\bv^{*(\ell)}\|^2=O(d_\ell)$.
\end{proof}
For the next step we denote $h(\bA)=\E_{\psi_1}\log \cZ_t/n$, as a function of all the elements $A_{\mu,i}$ of $\bA_\mu$ for $1\le i\le k$ and $1\leq\mu\leq n$.
\begin{lemma}
     There exists a non-negative constant $C(f,\varphi$) such that
  \begin{align*}
      \mathbb{E}\Big(\frac{1}{n}h-\EE_{\bA}\frac{1}{n}h\Big)^2\leq \frac{C(f,\varphi)}{n}\,.
  \end{align*}
\end{lemma}
\begin{proof}
     We denote by $\bA^\prime$ a vector such that $A^\prime_{\mu,i}=\bA_{\mu,i}$ for $\mu\neq\nu$, $i\neq j$ and $A^\prime_{\nu,j}$ is a random variable with distribution $P_A$, independent of all others.
    According to Proposition~\ref{prop:efron-stein} it is sufficient to prove that 
    \begin{align}\label{ineq:bound_dif_h}
        |h(\bA^\prime)-h(\bA)|<\frac{C(f,\varphi)}n\,.
    \end{align}
    If we denote by $H$ (and $H^\prime$) the Hamiltonian corresponding to $\cZ_t$ (and $\cZ_t$ with $\bA^\prime$), i.e.
    \begin{align*}
        &H=-\sum_{\mu=1}^n\log P_{out}(f(S_{t\mu};\bA_\mu)+\sqrt{\Delta}Z_\mu|s_{t\mu}),\\
        &H^\prime=-\sum_{\mu\neq\nu}\log P_{out}(f(S_{t\mu};\bA_\mu)+\sqrt{\Delta}Z_\mu|s_{t\mu})-\log P_{out}(f(S_{t\nu};\bA^\prime_\nu)+\sqrt{\Delta}Z_\nu|s_{t\nu}).
    \end{align*}
    one can see that
    \begin{align*}
        h(\bA^\prime)-h(\bA)
        =\frac{1}{n}\E_{\psi_1}\log \Big\langle e^{H-H^\prime}\Big\rangle_{H}
        \geq \frac{1}{n}\E_{\psi_1} \Big\langle H-H^\prime\Big\rangle_{H}\,,
    \end{align*}
    the last step being true due to Jensen inequality. On the other hand $h(\bA^\prime)-h(\bA)
                \leq n^{-1}\E_{\psi_1} \langle H-H^\prime\rangle_{H^\prime}$.
  We recall the definition  of $P_{  out}$ 
  \begin{align*}
    P_{  out}(y\mid x)=\int_{\mathbb{R}^p}\frac{dP_A(\bA)}{\sqrt{2\pi\Delta}}\exp\Big[-\frac{1}{2\Delta}(y-f(x;\bA))^2\Big]\,.
\end{align*}
Similarly, we obtain
  \begin{align*}
     H-H^\prime\geq\frac{1}{2\Delta}\Big\langle (f(S_{t\nu};\bA_\nu)-f(s_{t\nu};\tilde{\bA})+\sqrt{\Delta}Z_\nu)^2-(f(S_{t\nu};\bA_\nu^\prime)-f(s_{t\nu};\tilde{\bA})+\sqrt{\Delta}Z_\nu)^2\Big\rangle_{G}
\end{align*}
and
  \begin{align*}
     H-H^\prime\leq\frac{1}{2\Delta}\Big\langle (f(S_{t\nu};\bA_\nu)-f(s_{t\nu};\tilde{\bA})+\sqrt{\Delta}Z_\nu)^2-(f(S_{t\nu};\bA_\nu^\prime)-f(s_{t\nu};\tilde{\bA})+\sqrt{\Delta}Z_\nu)^2\Big\rangle_{G^\prime}\,,
\end{align*}
where $\langle\cdot\rangle_G$ (or with $G^\prime$) defined as
\begin{align}
   \langle\cdot\rangle_G=\dfrac{\int P_A(d\tilde{\bA})e^{-\frac{1}{2\Delta}(Y_{t\nu}-f(s_{t\nu};\tilde{\bA}))^2}(\cdot)}{\int P_A(d\tilde{\bA})e^{-\frac{1}{2\Delta}(Y_{t\nu}-f(s_{t\nu};\tilde{\bA}))^2}} 
\end{align}
or with $Y^\prime_{t\nu}$ where $\bA_\nu$ is changed to $\bA^\prime_\nu$. Since $f$ is bounded, we immediately obtain $|H-H^\prime|\leq C(f)(|Z_\mu|^2+1)$ and (\ref{ineq:bound_dif_h}). Now, due to the Proposition~\ref{prop:efron-stein}, the statement of the Lemma is proved.

\end{proof}

Finally, we prove the concentration with respect to $(\bW^{*(\ell)})_{\ell=1}^L$, $\ba^*$, and $\bv^*$.  For this we denote  $g\equiv g((\bW^{*(\ell)})_{\ell=1}^L,\ba^*,\bv^*):=\EE_{\psi_1,\psi_2}\log\cZ_t/n$.
\begin{lemma}
     There exists a non-negative constant $C(f,\varphi$) such that
  \begin{align*}
      \mathbb{E}\Big(g-\EE_{\psi_3}g\Big)^2\leq \frac{C(f,\varphi)}{d_m}\,.
  \end{align*}
\end{lemma}
\begin{proof}
    We use again Poincaré-Nash inequality and write
    \begin{align}
        \mathbb{E}\Big(g-\EE_{(\bW^{*(\ell)})_{\ell=1}^L,\ba^*,\bv^{*}}g\Big)^2\leq\sum_{\ell=1}^L\sum_{i_{\ell}=1}^{d_{\ell}}\sum_{j_{\ell}=1}^{d_{\ell-1}}\EE\Big(\frac{\partial g}{\partial W_{i_{\ell} j_{\ell}}^{*(\ell)}}\Big)^2+\sum_{i=1}^{d_L}\EE\Big(\frac{\partial g}{\partial a_{i}^{*}}\Big)^2+\sum_{j=1}^{d_{L-1}}\EE\Big(\frac{\partial g}{\partial v_{j}^{*}}\Big)^2\,.
    \end{align}
    
    \textit{Concentration with respect to $\bW^{*(\ell)}$}. 
To find the derivative of $g$ with respect to $W_{i_{\ell} j_{\ell}}^{*(\ell)}$ it would be convenient to calculate first the general derivative of $X_{\mu,p}^{(k)}$ for $k\geq \ell$
\begin{align*}
    \frac{\partial X_{p\mu}^{(k)}}{\partial W_{i_{\ell}j_{\ell}}^{*(\ell)}}=\frac{(\bB^{*[\ell+1:k]}_{\mu}\boldsymbol{\Phi}^{*(\ell-1)}_\mu)_{pi_{\ell}}X^{(\ell-1)}_{j_{\ell}\mu}}{\sqrt{ d_{\ell-1}}},
\end{align*}
where in the case of $k=\ell$ we set $\bB^{*[\ell+1:\ell]}=\mathbf{I}_{d_{\ell}}$. Then the derivative of $g$ is
\begin{align*}
    \frac{\partial g}{\partial W_{i_{\ell}j_{l}}^{*(\ell)}}
    &=\frac{1}{n}\sum_{\mu=1}^n\EE_{\psi_1,\psi_2}\Big[\Big\langle\frac{P_{\mu}^y}{P_\mu}\Big\rangle f^\prime(S_{t\mu};\bA_\mu)\Big(\frac{\sqrt{1-t}(\ba^{*\intercal}\bB^{*[\ell+1:L]}_{\mu}\boldsymbol{\Phi}^{*(\ell-1)}_\mu)_{i_l}X^{(\ell-1)}_{j_{\ell}\mu}}{\sqrt{d_Ld_{\ell-1}}}\\
    &+\frac{\sqrt{t}\rho_L(\bv^{*\intercal}\bB^{*[\ell+1:L-1]}_{\mu}\boldsymbol{\Phi}^{*(\ell-1)}_\mu)_{i_{\ell}}X^{(\ell-1)}_{j_{\ell}\mu}}{\sqrt{d_{L-1}d_{\ell-1}}}\Big)\Big]\,.
\end{align*}
For $\ell=L$ the second term is obviously absent, to reflect this, we set $\bB^{*[L+1:L-1]}_{\mu}=0$. It is worth to notice, that $\psi_1$ and $\psi_2$ contain all elements with dependency of $\mu$, which means that each term in the sum over $\mu$ will be the same, so we can remove $\frac{1}{n}\sum_{\mu}^n$ and write
\begin{align*}
    \frac{\partial g}{\partial W_{i_{\ell}j_{\ell}}^{*(\ell)}}
    &=\EE_{\psi_1,\psi_2}\Big[\Big\langle\frac{P_{\mu}^y}{P_\mu}\Big\rangle f^\prime(S_{t\mu};\bA_\mu)\Big(\frac{\sqrt{1-t}(\ba^{*\intercal}\bB^{*[\ell+1:L]}_{\mu}\boldsymbol{\Phi}^{*(\ell-1)}_\mu)_{i_\ell}X^{(\ell-1)}_{j_\ell\mu}}{\sqrt{d_Ld_{\ell-1}}}\\
    &+\frac{\sqrt{t}\rho_L(\bv^{*\intercal}\bB^{*[\ell+1:L-1]}_{\mu}\boldsymbol{\Phi}^{*(\ell-1)}_\mu)_{i_\ell}X^{(\ell-1)}_{j_\ell\mu}}{\sqrt{d_{L-1}d_{\ell-1}}}\Big)\Big]
\end{align*}
for fixed arbitrary $\mu$. The two resulting terms are of the very similar form, to avoid repetitive computations we will focus only on the first one, that we denote by $E$. Since later we will use the usual inequality $\EE(a+b)^2\leq 2\EE a^2+2\EE b^2$, it is sufficient to bound $\EE E^2$. One can see that by  estimating in a straightforward way we can achieve only the order of $O(1)$, which is not what we are aiming for. To reach the right order, we need another integration by parts and to be able to do so, we are introducing a circular interpolation which replaces $\bX_\mu^{(0)}$ by it's independent copy $\tilde{\bX}_\mu^{(0)}$
\begin{align}\label{eq:circ_inter_X}
    \bX_\mu^{(0)}(\tau)=\bX^{(0)}_\mu\sin\tau+\tilde{\bX}^{(0)}_\mu\cos\tau,
\end{align}
where $\tau\in[0,\pi/2]$. The nice properties of the new Gaussian variable $\bX_\mu^{(0)}(\tau)$ is that its  derivative with respect to $\tau$, denoted $\dot{\bX}_\mu^{(0)}(\tau)$, has simple correlations 
\begin{align}\label{eq:cov_X_tau}
\begin{split}
    &\EE_{\bX_\mu^{(0)},\tilde{\bX}_\mu^{(0)}}(X_{i\nu}^{(0)}(\tau)\dot{X}_{j\mu}^{(0)}(\tau))=0\,,\quad \EE_{\bX^{(0)}_\mu,\tilde{\bX}^{(0)}_\mu}(X^{(0)}_{i\nu}\dot{X}^{(0)}_{j\mu}(\tau))=\delta_{\mu\nu}\delta_{ij}\cos\tau\,.    
\end{split}
\end{align}
Respectively, in a recursive way we define $\bX_\mu^{(\ell)}(\tau)=\varphi(\bW^{*(\ell)}\bX^{(\ell-1)}_\mu(\tau)/\sqrt{d_{\ell-1})}$ and $\tilde{\bX}^{(\ell)}_\mu=\varphi(\bW^{*(\ell)}\tilde{\bX}^{(\ell-1)}_\mu/\sqrt{d_{\ell-1})}$ for $1\leq \ell\leq L$, as well as  $\boldsymbol{\alpha}^{*(\ell)}_\mu(\tau)=\bW^{*(\ell+1)}\bX^{(\ell)}_\mu(\tau)/\sqrt{d_{\ell}}$ and $\tilde{\boldsymbol{\alpha}}^{*(\ell)}_\mu=\bW^{*(\ell+1)}\tilde{\bX}^{(\ell)}_\mu/\sqrt{d_{\ell}}$.

First, let us deal with $\ell>1$. In such a case $\bX^{(\ell-1)}_\mu=\varphi(\boldsymbol{\alpha}^{*(\ell-2)}_\mu)$, which can be rewritten us
\begin{align}
   \varphi(\boldsymbol{\alpha}^{*(\ell-2)}_\mu)= \int_0^{\pi/2}d\tau \frac{d}{d\tau}{\varphi}(\boldsymbol{\alpha}^{*(\ell-2)}_\mu(\tau))+\varphi(\tilde{\boldsymbol{\alpha}}^{*(\ell-2)}_\mu)\,.
\end{align}
Since $\varphi$ is an odd function and $\tilde{\bX}^{(0)}_\mu$ is independent of every other variable in $g$, the term containing $\varphi(\tilde{\boldsymbol{\alpha}}^{*(\ell-2)}_\mu)$ will collapse to $0$, and we obtain immediately
\begin{align*}
   E= \EE_{\psi_1,\psi_2,\tilde{\bX}_\mu^{(0)}}\int_0^{\pi/2}\Big\langle\frac{P_{\mu}^y}{P_\mu}\Big\rangle f^\prime(S_{t\mu};\bA_\mu)\frac{\sqrt{1-t}(\ba^{*\intercal}\bB^{*[\ell+1:L]}_{\mu}\boldsymbol{\Phi}^{*(\ell-1)}_\mu)_{i_\ell}}{\sqrt{d_Ld_{\ell-1}}}\frac{d}{d\tau}{\varphi}(\alpha^{*(\ell-2)}_{j_\ell\mu}(\tau))d\tau. 
\end{align*}
Let us have  closer look at the derivative of ${\varphi}(\alpha^{*(\ell-2)}_{j_\ell\mu}(\tau))$. Using the chain rule and (\ref{eq:part_der_X_X}) we get
\begin{align}
   \frac{d}{d\tau}{\varphi}(\alpha^{*(\ell-2)}_{j_\ell\mu}(\tau))=\sum_{p=1}^{d_0}\frac{\partial X^{(\ell-1)}_{j_\ell\mu}(\tau)}{\partial X_{p\mu}^{(0)}(\tau)} \dot{X}_{p\mu}^{(0)}(\tau)=\sum_{p=1}^{d_0}B^{*[1:\ell-1]}_{j_\ell p,\mu}(\tau)\dot{X}_{p\mu}^{(0)}(\tau),
\end{align}
where for any $\ell_1,\ell_2\in [1,\ldots,L]$ matrix $\bB^{*[\ell_1:\ell_2]}_\mu(\tau)$ defined the same way as $\bB^{*[\ell_1:\ell_2]}_\mu$, with $\bX_\mu^{(0)}$ replaced by $\bX_\mu^{(0)}(\tau)$. With the help of (\ref{eq:cov_X_tau}) we  integrate $E$ by parts with respect to Gaussian parameter $\dot{X}_{p\mu}^{(0)}(\tau)$.
\begin{align}
E&=\sum_{p=1}^{d_0}\EE_{\psi,\tilde{\bX}_\mu^{(0)}}\int_0^{\pi/2}\cos\tau\Big[\Big\langle\frac{P_{\mu}^y}{P_\mu}\Big\rangle f^\prime(S_{t\mu};\bA_\mu)\frac{\sqrt{1-t}(\ba^{*\intercal}\bB^{*[\ell+1:L]}_{\mu}\boldsymbol{\Phi}^{*(\ell-1)}_\mu)_{i_\ell}}{\sqrt{d_Ld_{\ell-1}}}\Big]^{\prime}_{X_{p\mu}^{(0)}}B^{*[1:\ell-1]}_{\mu,j_\ell p}(\tau),\nonumber\\
    &=I_1+I_2+I_3,
\end{align}
    where we denoted
    \begin{align*}
        I_1& :=\EE_{\psi_1,\psi_2,\tilde{\bX}^{(0)}_\mu}\int_0^{\pi/2}\cos\tau\sum_{p=1}^{d_0}\Big(-\Big\langle\frac{P_{\mu}^y}{P_{\mu}}\Big\rangle^2f^\prime(S_{t\mu};\bA_\mu)^2+\Big\langle\frac{P_{\mu}^{yy}}{P_{\mu}}\Big\rangle f^\prime(S_{t\mu};\bA_\mu)^2\\
        &+\Big\langle\frac{P_{\mu}^y}{P_{\mu}}\Big\rangle f^{\prime\prime}(S_{t\mu};\bA_\mu)\Big)\frac{\partial S_{t\mu}}{\partial X^{(0)}_{p\mu}}\frac{\sqrt{1-t}(\ba^{*\intercal}\bB^{*[\ell+1:L]}_{\mu}\boldsymbol{\Phi}^{*(\ell-1)}_\mu)_{i_\ell}}{\sqrt{d_Ld_{\ell-1}}}B^{*[1:\ell-1]}_{j_\ell p,\mu}(\tau),\\
        I_2&:=\EE_{\psi_1,\psi_2,\tilde{\bX}^{(0)}_\mu}\int_0^{\pi/2}\cos\tau\sum_{p=1}^{d_0}\Big(-\Big\langle\frac{P_{\mu}^y}{P_{\mu}}\Big\rangle \Big\langle\frac{P_{\mu}^{x}}{P_{\mu}}\frac{\partial s_{t\mu}}{\partial X^{(0)}_{p\mu}}\Big\rangle +\Big\langle\frac{P_{\mu}^{yx}}{P_{\mu}}\frac{\partial s_{t\mu}}{\partial X^{(0)}_{p\mu}}\Big\rangle \Big)\\
        &\times f^{\prime}(S_{t\mu};\bA_\mu)\frac{\sqrt{1-t}(\ba^{*\intercal}\bB^{*[\ell+1:L]}_{\mu}\boldsymbol{\Phi}^{*(\ell-1)}_\mu)_{i_\ell}}{\sqrt{d_Ld_{\ell-1}}}B^{*[1:\ell-1]}_{j_\ell p,\mu}(\tau),\\
        I_3&:=\sum_{p=1}^{d_0}\EE_{\psi_1,\psi_2,\tilde{\bX}_\mu}\int_0^{\pi/2}\cos\tau\Big\langle\frac{P_{\mu}^y}{P_\mu}\Big\rangle f^\prime(S_{t\mu};\bA_\mu)\nonumber\\
        &\qquad\qquad\qquad\times\Big[\frac{\sqrt{1-t}(\ba^{*\intercal}\bB^{*[\ell+1:L]}_{\mu}\boldsymbol{\Phi}^{*(\ell-1)}_\mu)_{i_\ell}}{\sqrt{d_Ld_{\ell-1}}}\Big]^{\prime}_{X_{p\mu}^{(0)}}B^{*[1:\ell-1]}_{j_\ell p,\mu}(\tau).
    \end{align*}
   We start with the term $I_1$. The first bracket can be bounded right away with Lemma~\ref{lem:bound_P_out_der}, and since $Z_\mu$ is Gaussian and independent of rest of $I_1$ we also bound $\EE_{\psi_1} |Z_\mu|^k$ with a constant.  The derivative of $S_{t\mu}$ is easily found with the help of (\ref{eq:part_der_X_X}), where after some simplification we have
   \begin{align*}
       |I_1|\leq\frac{ C(f)}{\sqrt{d_Ld_{\ell-1}}}\EE_{\psi,\tilde{\bX}_\mu^{(0)}}\int_{0}^{\pi/2}&\Big|\frac{1}{\sqrt{d_L}}\big(\ba^{*\intercal}\bB^{*[1:L]}_\mu\bB^{*[1:\ell-1]\intercal}_\mu(\tau)+\\
       &\frac{1}{\sqrt{d_{L-1}}}\bv^{*\intercal}\bB^{*[1:L-1]}_\mu\bB^{*[1:\ell-1]\intercal}_\mu(\tau)\big)_{j_\ell}(\ba^{*\intercal}\bB^{*[\ell+1:L]}_{\mu}\boldsymbol{\Phi}^{*(\ell-1)}_\mu)_{i_\ell}\Big|d\tau.
   \end{align*}
   Separating two terms and using properties of operator norm several times as well as Jensen and Cauchy-Swartz inequalities, we have
   \begin{align*}
       \EE\sum_{i_\ell,j_\ell}|I_1|^2&\leq \frac{C(f)}{d_{\ell-1}}\int\EE\Big[\frac{\|\ba^{*}\|^4}{d_L^2}\|\bB^{*[1:L]}_\mu\|^2\|\bB^{*[1:\ell-1]}_\mu(\tau)\|^2\nonumber\\
       &+\frac{\|\ba^{*}\|^2\|\bv^{*}\|^2}{d_L d_{L-1}}\|\bB^{*[1:L-1]}_\mu\|^2\|\bB^{*[1:\ell-1]}_\mu(\tau)\|^2\Big]\|\bB^{*[\ell+1:L]}_{\mu}\|^2\|\boldsymbol{\Phi}^{*(\ell-1)}_\mu\|^2\leq\frac{C(f,\varphi)}{d_{\ell-1}}.
   \end{align*}
    Let us move to $I_2$. Thanks to \eqref{eq:part_der_x_X}, the derivative of $s_{t\mu}$ has again two terms (one with $\ba$ and second with $\bv$, since they are of the same form, we will look only at the first one) 
   \begin{align}\label{eq:der_W_I2}
       |I_2|&\leq \frac{C(f)}{\sqrt{d_{\ell-1}}}\EE_{\psi_1,\psi_2,\tilde{\bX}^{(0)}_\mu}\int_0^{\pi/2}\Big(\Big\langle\Big|\frac{P_{\mu}^y}{P_{\mu}}\Big|\Big\rangle \Big\langle\Big|\frac{P_{\mu}^{x}}{P_{\mu}}\Big|\frac{|(\ba^{\intercal}\bB^{[1:L]}_\mu\bB^{*[1:\ell-1]\intercal}_\mu(\tau))_{j_\ell}|}{\sqrt{d_{L}}}\Big\rangle\\ &+\Big\langle\Big|\frac{P_{\mu}^{yx}}{P_{\mu}}\Big|\frac{|(\ba^{\intercal}\bB^{[1:L]}_\mu\bB^{*[1:\ell-1]\intercal}_\mu(\tau))_{j_\ell}|}{\sqrt{d_{L}}}\Big\rangle \Big)
        \frac{|(\ba^{*\intercal}\bB^{*[\ell+1:L]}_{\mu}\boldsymbol{\Phi}^{*(\ell-1)}_\mu)_{i_\ell}|}{\sqrt{d_L}}+\ldots.\nonumber
   \end{align}
    After using Lemma \ref{lem:bound_P_out_der} and Jensen inequality we get 
   \begin{align}
       \EE\sum_{i_\ell,j_\ell}|I_2|^2&\leq \frac{C(f)}{d_{\ell-1}}\EE\int_0^{\pi/2} \Big\langle\frac{\|\ba^{\intercal}\bB^{[1:L]}_\mu\bB^{*[1:\ell-1]\intercal}_\mu(\tau)\|^2}{d_{L}}\Big\rangle
        \frac{\|\ba^{*\intercal}\bB^{*[\ell+1:L]}_{\mu}\boldsymbol{\Phi}^{*(\ell-1)}_\mu\|^2}{d_L}(1+|Z_\mu|^2)^2d\tau
        \nonumber\\
        &\quad\quad\quad\quad+\ldots.
   \end{align}
   Lastly, separating quenched from annealed terms with the help of Cauchy-Swartz, we can get rid of the Gibbs bracket by  using again Jensen and Nishimori. This will leave as with moments of operator norms of  $\|d_L^{-1/2}\ba^*\|$, $\|\bB\|$, etc., since all of them are of order $O(1)$ we achieve the desired bound
   \begin{align}
       \EE\sum_{i_\ell,j_\ell}|I_2|^2\leq \frac{C(f)}{d_{\ell-1}}.
   \end{align}
  The final part of $E$ we bound right away with
  \begin{align}\label{eq:der_W_I3}
      |I_3|\leq \frac{C(f)}{\sqrt{d_{\ell-1}}} \EE_{\psi_1,\psi_2,\tilde{\bX}_\mu}\int_0^{\pi/2}\Big|\sum_{p=1}^{d_0}\Big[\frac{(\ba^{*\intercal}\bB^{*[\ell+1:L]}_{\mu}\boldsymbol{\Phi}^{*(\ell-1)}_\mu)_{i_\ell}}{\sqrt{d_L}}\Big]^{\prime}_{X_{p\mu}^{(0)}}B^{*[1:\ell-1]}_{\mu,j_\ell p}(\tau)\Big|d\tau.
  \end{align}
  Variable $X^{(0)}_{\mu,p}$ appears in each $\boldsymbol{\Phi}^{*(\cdot)}_\mu$ inside $\bB^{*[\ell+1:L]}_{\mu}$, so the derivative on the r.h.s. will be the sum of $L-\ell+1$ terms. The generic $k$-th term of this sum is of form
  \begin{align*}
   \Big(\underbrace{\frac{\ba^{*\intercal}}{\sqrt{d_L}}   \boldsymbol{\Phi}^{*(L-1)}_\mu\frac{\bW^{*(L)}}{\sqrt{d_{L-1}}}\boldsymbol{\Phi}^{*(L-2)}_\mu\ldots\frac{\bW^{*(k+2)}}{\sqrt{d_{k+1}}}}_{\mathbf{e}}&\frac{\partial\boldsymbol{\Phi}^{*(k)}_\mu}{\partial X_{p\mu}^{(0)}}\underbrace{\frac{\bW^{*(k+1)}}{\sqrt{d_{k}}}\ldots\boldsymbol{\Phi}^{*(\ell-1)}_\mu}_{\mathbf{E}}\Big)_{i_\ell}\\
   &=\sum_{i=1}^{d_{k+1}}\mathbf{e}_i(\bar{\bB}^{*[1:k+1]}_\mu)_{ip}\mathbf{E}_{ii_\ell}.
  \end{align*}
  Here, $\bar{\bB}^*_\mu$ refers to the same expression as $\bB^*_\mu$, where we take $\varphi''$ instead of $\varphi'$. Since the number of such terms is finite ($L-\ell+1$) we can focus only on one $k\in[\ell-1,\ldots,L-1]$, as the others are bounded in the same way (we denote them with $\ldots$ in the following). Plugging the previous equation into \eqref{eq:der_W_I3} we have
  \begin{align*}
      |I_3|&\leq  \frac{C(f)}{\sqrt{d_{\ell-1}}} \EE_{\psi,\tilde{\bX}_\mu}\int_0^{\pi/2}\Big|\sum_{i=1}^{d_{k+1}}\mathbf{e}_i(\bar{\bB}^{*[1:k+1]}_\mu\bB^{*[1:\ell-1]\intercal}_{\mu}(\tau))_{ij_\ell}\mathbf{E}_{ii_\ell}\Big|d\tau+\ldots\\
      &=\frac{C(f)}{\sqrt{d_{\ell-1}}} \EE_{\psi,\tilde{\bX}_\mu}\int_0^{\pi/2}\Big|(\mathbf{E}^\intercal \text{diag}(\mathbf{e})\bar{\bB}^{*[1:k+1]}_\mu\bB^{*[1:\ell-1]\intercal}_{\mu}(\tau))_{i_\ell j_\ell}\Big|d\tau+\ldots.
  \end{align*}
  With Jensen inequality we obtain
  \begin{align*}
      &\EE\sum_{i_\ell j_\ell}|I_3|^2\leq\\
      &\frac{C(f)}{d_{\ell-1}}\EE\int_0^{\pi/2}\Tr\Big(\mathbf{E}^\intercal \text{diag}(\mathbf{e})\bar{\bB}^{*[1:k+1]}_\mu\bB^{*[1:\ell-1]\intercal}_{\mu}(\tau)\bB^{*[1:\ell-1]}_{\mu}(\tau) \bar{\bB}^{*[1:k+1]\intercal}_\mu \text{diag}(\mathbf{e})\mathbf{E}\Big)d\tau+\ldots\\
      &\leq \frac{C(f)}{d_{\ell-1}}\EE\int_0^{\pi/2}\|\mathbf{E}\|^2\|\bar{\bB}^{*[1:k+1]}_\mu\bB^{*[1:\ell-1]\intercal}_{\mu}(\tau)\|^2\Tr \,\text{diag}(\mathbf{e})^2d\tau+\ldots\\
      &=\frac{C(f)}{d_{\ell-1}}\EE\int_0^{\pi/2}\|\mathbf{E}\|^2\|\bar{\bB}^{*[1:k+1]}_\mu\bB^{*[1:l-1]\intercal}_{\mu}(\tau)\|^2\|\mathbf{e}\|^2d\tau+\ldots.
  \end{align*}
  To come to this conclusion we used several times the fact that for a symmetric semidefinite positive matrix $\bA$ and any matrix $\bB$ we have $|\Tr(\bA\bB)|\leq \|\bB\|\Tr \bA$. Now we can use Cauchy-Swartz and notice that moments of  $\|\boldsymbol{\Phi}^{*(\cdot)}_\mu\|$,  $\|\frac{\bW^{*(\cdot)}}{\sqrt{d_{\cdot}}}\|$, and $\|\frac{\ba^{*\intercal}}{\sqrt{d_L}}\|$ are all of order $O(1)$. Which gives $ \EE\sum_{i_\ell j_\ell}|I_3|^2=O(d^{-1}_{\ell-1})$.  
  
  To conclude the proof of the Lemma we need to deal with the derivative with respect to $\bW^{*(1)}$. This terms are very similar to what  we had before 
  \begin{align*}
          \frac{\partial g}{\partial W_{i_1j_1}^{*(1)}}
    &=\EE_{\psi_1,\psi_2}\Big[\Big\langle\frac{P_{\mu}^y}{P_\mu}\Big\rangle f^\prime(S_{t\mu};\bA_\mu)\Big(\frac{\sqrt{1-t}(\ba^{*\intercal}\bB^{*[2:L]}_{\mu}\boldsymbol{\Phi}^{*(0)}_\mu)_{i_1}X^{(0)}_{j_1\mu}}{\sqrt{d_Ld_{0}}}\\
    &+\frac{\sqrt{t\rho_L}(\bv^{*\intercal}\bB^{*[2:L-1]}_{\mu}\boldsymbol{\Phi}^{*(0)}_\mu)_{i_1}X^{(0)}_{j_1\mu}}{\sqrt{d_{L-1}d_{0}}}\Big)\Big].
  \end{align*} 
  The only difference is that now we have a factor $X^{(0)}_{j_1\mu}$ which is Gaussian itself and allows us to integrate by parts without introducing the circular interpolation. The calculations themselves will be identical. 

  \textit{Concentration with respect to $\ba^*$ and $\bv^{*}$}. Derivatives of $g$ with respect to these two vectors are almost identical, so we demonstrate here only the part with $\frac{\partial g}{\partial a^{*}_i}$.
  \begin{align*}
      \frac{\partial g}{\partial a^{*}_i}\!=\!\dfrac{1}{n}\!\sum_{\mu=1}^n\EE_{\psi_1,\psi_2}\Big[\Big\langle\frac{P_{\mu}^y}{P_\mu}\Big\rangle f^\prime(S_{t\mu};\bA_\mu)\frac{\sqrt{1-t}X^{(L)}_{i\mu}}{\sqrt{d_L}}\Big]\!=\!\EE_{\psi_1,\psi_2}\Big[\Big\langle\frac{P_{\mu}^y}{P_\mu}\Big\rangle f^\prime(S_{t\mu};\bA_\mu)\frac{\sqrt{1-t}X^{(L)}_{i\mu}}{\sqrt{d_L}}\Big].
  \end{align*}
  Following the lines of the first part of this Lemma we introduce the circular interpolation (\ref{eq:circ_inter_X})-(\ref{eq:cov_X_tau}) and rewrite the previous expression as
    \begin{align}
      \frac{\partial g}{\partial a^{*}_i}=\EE_{\psi_1,\psi_2}\int_0^{\pi/2}\Big[\Big\langle\frac{P_{\mu}^y}{P_\mu}\Big\rangle f^\prime(S_{t\mu};\bA_\mu)\frac{\sqrt{1-t}\sum_{p=1}^{d_0}B^{*[1:L]}_{ip}(\tau)\dot{X}_{p\mu}^{(0)}(\tau)}{\sqrt{d_L}}\Big]d\tau.
  \end{align}
  Proceeding with the integration by parts with respect to $\dot{X}_{\mu,p}^{(0)}(\tau)$ we obtain
  \begin{align}
       \frac{\partial g}{\partial a^{*}_i}=\EE_{\psi_1,\psi_2}\int_0^{\pi/2}\sum_{p=1}^{d_0}\cos\tau\Big[\Big\langle\frac{P_{\mu}^y}{P_\mu}\Big\rangle f^\prime(S_{t\mu};\bA_\mu)\Big]^\prime_{X^{(0)}_{p\mu}}\frac{\sqrt{1-t}B^{*[1:L]}_{\mu,ip}(\tau)}{\sqrt{d_L}}d\tau=I_1+I_2,
  \end{align}
  where we again denote
    \begin{align*}
        I_1&=\EE_{\psi_1,\psi_2,\tilde{\bX}^{(0)}_\mu}\int_0^{\pi/2}\cos\tau\sum_{p=1}^{d_0}\Big(-\Big\langle\frac{P_{\mu}^y}{P_{\mu}}\Big\rangle^2f^\prime(S_{t\mu};\bA_\mu)^2+\Big\langle\frac{P_{\mu}^{yy}}{P_{\mu}}\Big\rangle f^\prime(S_{t\mu};\bA_\mu)^2\\
        &+\Big\langle\frac{P_{\mu}^y}{P_{\mu}}\Big\rangle f^{\prime\prime}(S_{t\mu};\bA_\mu)\Big)\frac{\partial S_{t\mu}}{\partial X^{(0)}_{p\mu}}\frac{\sqrt{1-t}B^{*[1:L]}_{\mu,ip}(\tau)}{\sqrt{d_L}},\\
        I_2&=\EE_{\psi_1,\psi_2,\tilde{\bX}_\mu}\int_0^{\pi/2}\cos\tau\sum_{p=1}^{d_0}\Big(-\Big\langle\frac{P_{\mu}^y}{P_{\mu}}\Big\rangle \Big\langle\frac{P_{\mu}^{x}}{P_{\mu}}\frac{\partial s_{t\mu}}{\partial X^{(0)}_{p\mu}}\Big\rangle +\Big\langle\frac{P_{\mu}^{yx}}{P_{\mu}}\frac{\partial s_{t\mu}}{\partial X^{(0)}_{p\mu}}\Big\rangle \Big)\\
        &\times f^{\prime}(S_{t\mu};\bA_\mu)\frac{\sqrt{1-t}B^{*[1:L]}_{\mu,ip}(\tau)}{\sqrt{d_L}}\,.
    \end{align*}
    Using (\ref{eq:part_der_X_X}) to calculate derivative of $S_\mu$ and bounding ratios of $P_\mu$ with Lemma~\ref{lem:bound_P_out_der}, we obtain
    \begin{align*}
        |I_1|\leq \frac{C(f)}{\sqrt{d_L}}\EE_{\psi_1,\psi_2,\tilde{\bX}_\mu}\int_0^{\pi/2}\Big(\frac{\sqrt{1-t}\ba^{*\intercal}}{\sqrt{d_L}}\bB^{*[1:L]}_\mu \bB^{*[1:L]\intercal}_\mu(\tau)+\frac{\sqrt{t}\rho_L\bv^{*\intercal}}{\sqrt{d_{L-1}}}\bB^{*[1:L-1]}_\mu \bB^{*[1:L]\intercal}_\mu(\tau)\Big)_{i}.
    \end{align*}
    Taking the sum over $i$ of squared latter expression we obtain
        \begin{align*}
        \sum_{i}|I_1|^2\leq \frac{C(f,t)}{d_L}\EE\int_0^{\pi/2}d\tau\Big(\Big\|\frac{\ba^{*\intercal}}{\sqrt{d_L}}\bB^{*[1:L]}_\mu\bB^{*[1:L]}_\mu(\tau)\Big\|^2+\Big\|\frac{\bv^{*\intercal}}{\sqrt{d_{L-1}}}\bB^{*[1:L-1]}_\mu\bB^{*[1:L]}_\mu(\tau)\Big\|^2\Big).
    \end{align*}
    Again, since $\EE\|\ba^*\|^2/d_L$ and $\EE\|\bB^{*[\ell_1:\ell_2]}_\mu(\tau)\|$ are all of order $O(1)$, we obtain that $\sum_{i}|I_1|^2\leq C(f,\varphi)/d_{L}$.
    Term $I_2$ is handled very similar to \eqref{eq:der_W_I2}, resulting in the bound $\sum_{i}|I_2|^2\leq C(f,\varphi)/d_{L}$, which concludes the proof.
\end{proof}

\end{document}